\documentclass[11pt]{amsart}
\usepackage{amsmath, latexsym, amssymb, times}
\usepackage[colorlinks=true,pagebackref,hyperindex]{hyperref}
\newtheorem{theorem}{Theorem}[section]
\newtheorem{proposition}[theorem]{Proposition}
\newtheorem{corollary}[theorem]{Corollary}
\newtheorem{cor}[theorem]{Corollary}
\newtheorem{prop}[theorem]{Proposition}

\newtheorem{notation}[theorem]{Notation}
\newtheorem{example}[theorem]{Example}
\newtheorem{remark}[theorem]{Remark}
\newtheorem{lemma}[theorem]{Lemma}
\newtheorem{definition}[theorem]{Definition}

\newenvironment{defn}{\begin{definition}\rm}{\end{definition}}
\newenvironment{notn}{\begin{notation}\rm}{\end{notation}}
\newenvironment{ex}{\begin{example}\rm}{\end{example}}

\newcommand{\R}{\mathbb{R}}

\newcommand{\N}{\mathbb{N}}

\newcommand{\vv}{\mathcal{RV}}
\newcommand{\proj}{\mathrm{Proj}}
\newcommand{\m}{\mathfrak{m}}
\newcommand{\fa}{\mathfrak{a}}
\newcommand{\fb}{\mathfrak{b}}
\newcommand{\n}{\mathfrak{n}}
\newcommand{\q}{\mathfrak{q}}
\newcommand{\p}{\mathfrak{p}}
\newcommand{\II}{\overline{I}}
\newcommand{\IC}{\mathbb{C}}

\newcommand{\IA}{\mathbb{A}}
\newcommand{\cO}{\mathcal O}

\newcommand{\ord}{\mathrm{ord}}
\newcommand{\spec}{\operatorname{Spec}}
\newcommand{\reg}{\operatorname{Reg}}
\newcommand{\vvar}{\mathrm{Var}}

\newcommand{\xvec}[1]{\ensuremath{x_{1}, \ldots, x_{#1}}}

\newcommand{\Projan}{\operatorname{Projan}}
\newcommand{\Specan}{\operatorname{Specan}}

\begin{document}
\title{Weak subintegral closure of ideals}
\author{Terence Gaffney}
\author{Marie A. Vitulli}

\address{Department of Mathematics, Northeastern University, 567 Lake Hall, 360 Huntington Ave, Boston, MA 02115-5095} 
\email{gaff@research.neu.edu}

\address{Department of Mathematics, 1222-University of Oregon, Eugene, OR  97403-1222}
\email{vitulli@uoregon.edu}

\thanks{The authors would like to thank the AMS, Caroline Grant Melles, and the late Ruth Michler for sponsoring and organizing the special session on Singularities in Algebraic and Analytic Geometry in San Antonio (1999) that brought them together in this project. They would also like to thank Craig Huneke for helpful conversations.}
\date{September 12, 2008}
\maketitle

\begin{abstract} We describe some basic facts about the weak subintegral closure of ideals  in both the algebraic and complex-analytic settings.  We focus on the analogy between  results on the integral closure of ideals and modules and the weak subintegral closure of an ideal.  We start by giving a geometric interpretation of the Reid-Roberts-Singh criterion for when an element is weakly subintegral over a subring. We give new characterizations of the weak subintegral closure of an ideal.  We associate with an ideal $I$ of a ring $A$  an ideal $I_>$, which consists of all elements of $A$ such that $v(a)>v(I)$, for all  Rees valuations $v$ of $I$. The ideal $I_>$ plays an important role in conditions from stratification theory such as Whitney's condition A and Thom's condition $A_f$ and  is contained in every reduction of $I$.  We close with a valuative criterion for when an element is in the weak subintegral closure of an ideal.  For this, we introduce a new closure operation for a pair of modules, which we call relative closure.

\end{abstract}

\section{Introduction}
The purpose of this paper
is to describe some basic facts about the weak subintegral closure of ideals  in both the algebraic and complex-analytic settings.  We focus on the analogy between  results on the integral closure of ideals and modules and the weak subintegral closure of an ideal. 

Since we are interested in the characteristic zero case, we blur the distinction between the operations of weak normalization and seminormalization and the related conditions.

We first sketch a brief history of our subject.  In 1967 the operation of \textit{weak normalization} was introduced in the complex analytic setting  by Andreotti and Norguet in \cite{AN}.  An analytic space is \textit{weakly normal} (that is, equal to its weak normalization) if every continuous complex-valued function that is holomorphic off the singular locus is globally holomorphic. In 1969 the weak normalization of an abstract scheme was studied by Andreotti and Bombieri \cite{AB}.   Traverso \cite{Tr} introduced the
operation of \textit{seminormalization} for integral ring extensions the following year and showed that the seminormalization of
a ring $A$ in an integral extension $B$ is the largest subring whose prime spectrum is in bijective correspondence with Spec($A$) and with isomorphic residue fields.   Traverso then looked at Noetherian reduced rings and in that context defined a seminormal ring as one
that equals its seminormalization in its normalization.  Traverso's construction of the seminormalization was local in nature and  involved ``gluing" over the various primes ideals of $B$ that lie over a single prime ideal of $A$.   Hamann \cite{Ha} later demonstrated that $A$ is seminormal in $B$ if and only if every element $b \in B$ such that $b^2, b^3 \in A$ is actually in $A$. Traverso showed that with the
   additional assumption that the normalization is finite, then the canonical homomorphism
    $\mathrm{Pic } \, A \rightarrow \mathrm{Pic }  \, A[T] $ is an isomorphism if and only if $A$ is seminormal.
     A decade after Traverso first introduced the operation of seminormalization,
        Swan \cite{Sw} called an extension of rings $A\subset B$ \textit{subintegral }if $B$ is
         integral over $A$ and the inclusion induces a bijection of prime spectrums and
          isomorphisms of the residue fields.   Swan made a small but significant modification of Hamann's
           characterization of seminormal rings by declaring that a reduced ring $A$ is seminormal
            in an extension ring $B$ if  whenever $b, c  \in B$ with $b^3 = c^2$, there exists an
             element $a \in A$ such that $a^2 = b \mbox{ and } a^3 = c$.   For a reduced ring
              whose quotient ring is a product of fields, the two notions agree, but don't in
              general. Swan was able to show that, in general,  the homomorphism $\mathrm{Pic } \, A \rightarrow \mathrm{Pic }  \, A[T] $ is an
     isomorphism if and only if $A_{\mathrm{red}}$ is seminormal.  Swan also constructed the
     seminormalization of an arbitrary reduced ring by mimicking the construction of the
     algebraic closure of a field.

 The notions of weak subintegral closure and weak normalization are closely related.  An extension of rings $A\subset B$ is \textit{weakly subintegral }if $B$ is
integral over $A$, the inclusion induces a bijection of prime spectrums, and
purely inseparable extensions of the residue fields. An element $b$ in $B$ is weakly subintegral over $A$ if $A \subset A[b]$ is a weakly subintegral extension of rings. The weak normalization of $A$ in $B$ is the largest weakly subintegral extension of $A$ in $B$. In these notes we work over characteristic zero exclusively.  In this context, the seminormalization of $A$ in $B$ coincides with the weak normalization of $A$ in $B$.
We will use the latter terminology in the subsequent sections of this paper.

Now we consider the weak subintegral closure, ${}^*I$, of an ideal $I$. We will use the definition proposed by Vitulli and Leahy \cite{VL}. This is described in detail in section 2. For now, we note an important link between weak normalization of a graded ring and the weak subintegral closure of an ideal. Suppose that $A \subseteq B$ are rings, $I$ is an ideal in $A$, and $b \in B$. Then, $b$ is weakly subintegral over $I$ if and only if the element $bt \in B[t]$ is weakly subintegral over the Rees ring $A[It]$ ( \cite[Lemma 3.2]{VL}, for $n=1$). This connection is parallel to a connection in the theory of integral closure of ideals. If $A$ is a ring, $I$ is an ideal in $A$, $b\in A$ is integral over $I$ if and only if $bt\in A[t]$ is integral over the Rees ring $A[It]$. This relationship suggests developing other parts of the theory of the weak subintegral closure of an ideal in parallel with the theory of the integral closure of an ideal. Indeed, the parallelism can be taken further and applied to modules as well.  This is the starting point for our paper.

In section 2 we describe a characterization due to Reid, Roberts and Singh (\cite[Condition 1.3] {RRS}. They gave a criterion  for an element $b\in B$ to be an element
of the weak normalization of $A$ similar to the
criterion for $b$ to be in the normalization.   We give a geometric interpretation of their criterion in Propositions \ref{alg branched cover} and \ref{anal branched cover}, which explain the form taken by the equations appearing in their criterion. Roughly speaking, a rational function is in the integral closure of the local ring of a variety if and only if its graph can be embedded in a branched cover of the variety. Proposition \ref{anal branched cover}  shows that the function is in the weak normalization if the graph embeds in the part of the cover in which ``all the sheets come together."

A basic result in the theory of integral closure of ideals is that $h \in \overline{ I}$ if and only if the pullback of $h$ to the normalized blow-up by $I$ is in the pull back of $I$. We pursue this approach in section 3. In Theorem \ref{local anal char of *I}, we prove the analogous statement for weak subintegral closure for local analytic rings and in Theorem \ref{local alg char of *I}  we prove the algebraic version.

In section 4 we relate the weak subintegral closure of an ideal $I\subset A$ to the ideal $I_{>}$, which consists of all elements of $A$ such that $v(a)>v(I)$, for all  Rees valuations $v$ of $I$. The ideal $I_{>}$ plays an important role in conditions from stratification theory such as Whitney's condition A and Thom's condition $A_f$.  We show that  $I_{>}\subset {}^*I$ (Proposition \ref{prop:lantz}), and that any reduction $J$ of $I$ contains $I_{>}$ as well (Cor. \ref{cor:lantz}). Proposition \ref{prop:lantz} proves a generalization of a conjecture of Lanz.

If we restrict to 0-dimensional ideals $I$ in a Noetherian local ring and assume that the residue field is algebraically closed of characteristic 0, we can say more. If $J$ is a minimal reduction of $I$, then $J+I_{>}={}^*J$ (Theorem \ref{thm:weaknor of min red}), and if $I$ is generated by the minimal number of generators, then for every reduction $J$, we have ${}^*J=I$, otherwise, the set of all elements of $A$ weakly subintegral over all reductions $J$ of $I$ is precisely ${}^*I$ (Cor. \ref{ cor: *J = I}).
If $I$ is $\m$-primary
then the elements of $I_{>}$ are also known as elements of $A$ which are strictly dependent on $I$ and are denoted $I^{\dagger}$ (see section 3 of \cite{GK} or section 5 of \cite{GG}). 
There are valuative criteria for both the elements of $I$ and $I^{\dagger}$. In section 5 we develop such a criterion for the elements of ${}^*I$. The criterion is based on a new closure operation, relative closure (Definition \ref{rel closure anal} in the analytic case and \ref{rel closure alg} in the algebraic case). This closure operation is defined by a valuative criterion. Because weak subintegral dependence is connected with proving a projection is a homeomorphism, our analytic criterion is based on map germs from $\IC$ to $X \times X$, and even in the ideal case we are led to use pairs of modules in the relative closure operation.  The criterion is contained in Theorem \ref{al crit anal}. In turn, the algebraic criterion is based on local homomorphism of $\IC$-algebras from $A \otimes _{\IC} A \to \IC[[z]]$, where $A = \cO_{X,x}$.  In this setting, the field $k$ can be replaced by any algebraically closed field $k$.  Another, interesting approach to a valuative criterion is being developed by Holger Brenner (private correspondence).

As this introduction shows, the development of this paper is based on both the geometric/analytic and algebraic points of view. In writing the paper we have tried to incorporate both perspectives, as they each give valuable insight into the subject. We hope this style of writing will also make the contributions of each community more available to the other.

\section{Preliminaries}
    In this section we recall and build on the element-wise definition of weak subintegrality introduced by Reid, Roberts, and Singh  (\cite[Condition 1.3] {RRS}).  We present a geometric interpretation of their definition that gives additional insight into why such a system of equations occurs.  We finish the section by recalling the element-wise definition of weak subintegrality over an ideal introduced by Vitulli and Leahy in \cite{VL}.

Recall that an integral extension of rings $A \subset B$ is weakly subintegral if the induced map $\spec(B) \to \spec(A)$ is a bijection with purely inseparable extensions of the residue fields.  Note that if we work with finitely-generated reduced algebras over $\mathbb{C}$ (or any algebraically closed field of characteristic 0) or with the local ring of a complex-analytic space, 
it suffices to require that the induced map of prime spectrums is a bijection. 
In this case,  for each  $Q\ \in \spec( B)$,  the induced map from $\spec(B/Q)$ to $\spec(A/(Q \cap A))$ must be a bijection. 
Recall that the degree of the quotient field of $B/Q$ over the quotient field of $A/(Q\cap A)$ is equal to the number of preimages of a general point of 
Spec($A/(Q \cap A)$)
( cf. \cite[Proposition 3.17]{M}   in the algebraic case; in the analytic case when this number is 1
this is a corollary of the existence of universal denominators, e.g., see  \cite[Theorem 21]{Gu}).  It follows that if the number of generic preimages is 1, 
then the inclusions are isomorphisms, provided that $\dim V(Q)>0$.  Of course, if $\spec(B/Q)$ 
is $0$-dimensional then $Q$ must be a maximal ideal and the residue fields are just $\mathbb{C}$. So 
 the map still induces an isomorphism.

Let's return to the general situation. Given an extension of rings $A \subset B$, the weak normalization ${}^*_B A$ of  $A$ in $B$ is defined by
%: def of *A
\begin{equation}{}^*_B A = \{ b \in B \mid \forall \p \in \spec(A), \; \exists n \ge 0 \mbox{ such that } (b/1)^{e^n} \in A_{\p} + R(B_{\p}) \},
\end{equation}
where $R(B_{\p})$ is the Jacobson radical of $B_{\p}$ and $e$ is the characteristic exponent of the residue class field $\kappa(\p)$ of $A_{\p}$.  One can show that ${}^*_B A$ is  the set of all elements $b\in B$ such that $A\subset A[b]$ is a weakly subintegral extension. If $B$ is the normalization of the reduced Noetherian ring $A$, we refer to the weak normalization of  $A$ in $B$ as the weak normalization of  $A$ and denote it by ${}^*A$.

Reid, Roberts and Singh (\cite[Condition 1.3] {RRS} gave a criterion for an element $b\in B$ to be an element 
of the weak normalization of $A$ similar to the 
criterion for $b$ to be in the normalization.   We now recall their definition.

An element $b \in B$ is said to be
{\it weakly  subintegral over  A
} provided that there exist a nonnegative integer
 $q $ and  elements $a_i\in A \hskip .1cm (1\le{i}\le 2q+1$)
such that $b$ satisfies the equations
\begin{equation} \label{RRSeqns}
T^n+\sum_{i=1}^{n} {n \choose i} a_iT^{n-i}=0\quad (q+1\le n \le 2q+1).
\end{equation}

In the original characterization of weakly subintegral elements, a factor of $(-1)^i$ accompanied the coefficient $a_i$, but here we absorb that factor in $a_i$. Reid, Roberts, and Singh proved that $b \in B$ is weakly subintegral over $A$ if and only if $A \subseteq A[b]$ is a weakly subintegral extension (see \cite[Theorem 6.11] {RRS}).

There is an interesting geometric interpretation of this definition which seems to be new. It  provides some insight into the appearance of the string of equations (\ref{RRSeqns}).  We first make some general observations and establish some helpful notation.

For a polynomial $F(X,T) \in k[X_1, \ldots , X_m, T]$ that is monic in $T$ of degree $N$, let $Z(F) \subset \IA ^{m+1}$ denote the zeroes of $F$ and $ZZ(F) \subset \IA ^{m+1}$ denote the common zeroes of $F$ and its first $\ell = \lfloor \frac{N}{2} \rfloor$ derivatives with respect to $T$.  

%: lemma lem zeros
\begin{lemma} \label{lem zeros}
 Let $k$ be an algebraically closed field $k$ of characteristic 0,
 $F = F(X,T)$ be a polynomial in  $k[X_1, \ldots , X_m, T]$,
monic in $T$, $Z(F)$  and $ZZ(F) \subset Z(F)$ be as above.  Then, the restriction of the projection $p \colon \IA^{m+1} \to \IA^m $ onto the first $m$ factors to $ZZ(F)$ is a homeomorphism onto its image, which is closed in $\IA ^m$.
\end{lemma}

\begin{proof}
 Notice that for any $x \in \IA^m$ the equation  $F(x,T) = 0$ has at most one root in $k$ of multiplicity at least $\ell + 1$.
 For if $(x, t_1) \ne(x,t_2)$ are both such, then $(T-t_1)^{\ell + 1}(T-t_2)^{\ell + 1}$ divides $F(x,T)$, which is absurd, since $2(\ell + 1) > N$.  Since $(x,t) \in ZZ(F) \Leftrightarrow t$ is a root of $F(x,T)$ of multiplicity at least $\ell + 1$  the restriction of $p$  to $ZZ(F)$ is a homeomorphism onto its image.  Since the restriction of the projection to $Z(F)$ is a closed mapping, the image of $p(ZZ(F))$ is closed in $\IA ^m$.
\end{proof}

Let $k$ be an algebraically closed field $k$ of characteristic 0. Suppose $A = k[V]$ is the affine coordinate ring of an affine variety $V \subset \IA^m$ and $h \in {}^*A$.  First consider $h$ as a regular function on the normalization $\widetilde{V}$  of $V$.  Since $h$ is constant on the fibers of the projection $\pi \colon \widetilde{V} \to V$ we may regard $h$ as a continuous $k$-valued function on $V$.
Returning to the equations (\ref{RRSeqns}) of Reid, Roberts, and Singh, notice that the equation of degree $2q$ in the system is just an integer multiple of the derivative of the equation of degree $2q+1$ in the $T$ variable,
 and in fact the equation of degree $n$ is just an integer multiple of the derivative of order $2q+1-n$ of the top degree equation. This relationship was observed by L. G. Roberts in the proof of Lemma 4.1 of \cite{Rob}, but wasn't taken any further.  Suppose $h$ satisfies this system of equations. 
 Let $f_i \in k[X_1, \ldots , X_m]$ be a representative of $a_i \in A \; (i = 1, \ldots ,2q+1)$ and set $F(X,T) = T^{2q+1} + \sum_{i=1}^{2q+1} {2q+1 \choose i} f_i(X)T^{2q+1-i}$.  Applying Lemma \ref{lem zeros} to the hypersurface $Y = Z(F) \subset \IA ^{m+1}$ we see that $(x,h(x)) \in ZZ(F)$ and hence $F(x, T) = (T - h(x))^{q+1}G(x,T)$, where $G(x,T) \in k[T]$;  this holds for each $x \in V$.   Letting $\Gamma_h$ denote the graph of $h$, we have $p^{-1}(V) \cap ZZ(F) = \Gamma_h$,  where $p \colon \IA^{m+1} \to \IA ^m$ is the projection onto the first $m$ factors.

 \bigskip
 
 We would now like to give a geometric characterization of when a function on an affine algebraic variety $V$  becomes a regular function on its weak normalization of $V$.  We work over an algebraically closed field $k$ of characteristic 0.  There are (at least) two ways to approach this: one can consider a $k$-valued function defined globally on $V$ or a regular function on the normalization of $V$ considered as a rational function on $V$.  We take the latter approach.   In \cite{Vit} the second author characterized the $k$-valued functions on an affine variety $V$ without one-dimensional components that become regular on the weak normalization as those functions satisfying two conditions: every polynomial in $h$ with coefficients in the affine coordinate ring of $V$ is continuous (w.r.t the Zariski topology) and the graph of $h$ is closed in $V \times \IA$.  The latter condition plays a key role in our new characterization.
 
  For a rational function $h$ on $V$ let $\reg(h)$ denote the set of points where $h$ is regular and let 
  $\Gamma_h \subset \reg(h) \times \IA^1$ denote the graph of $h: \reg(h) \rightarrow k$.  Let $p \colon \IA ^{m+1} \to \IA ^m$  the projection onto the first $m$ factors.

%: prop   branched cover algebraic version
\begin{prop}  \label{alg branched cover} Let $k$ be an algebraically closed field of characteristic 0, $V \subset \IA^m$  an 
affine variety with affine coordinate ring $A$, and let $h$ be in the normalization of $A$. Then, $h$ is in the weak normalization of $A$ if and only if there exists a polynomial $F(X,T)$ in $k[X_1, \ldots , X_m, T]$ that is monic in  $T$ such that 
$$\Gamma_h  \subset p^{-1}(V) \cap ZZ(F)  .$$
\end{prop}

\begin{proof}  
Assume that $h \in {}^*A$.  Suppose $q \ge 0$ and $a_i \in A \; (1 \le i \le 2q+1)$ are such that $h$ satisfies the resulting equations (\ref{RRSeqns}).  Let $f_i \in k[X_1, \ldots , X_m]$ be a preimage of $a_i \in A$ for $(i = 1, \dots , 2q+1)$ and set $F(X,T) = T^{2q+1} + \sum_{i=1}^{2q+1} {2q+1 \choose i} f_i(X) T^{2q+1-i} $. As remarked earlier, we must have $p: Z(F) \rightarrow \IA ^m$ is a surjective finite morphism such that $p^{-1}(V) \cap ZZ(F) = \Gamma_h$. 

Conversely assume that there exists a polynomial $F = F(X,T) \in k[X_1, \ldots , X_m, T]$ that is 
monic in  $T$ such that $$ \Gamma_h \subset p^{-1}(V) \cap ZZ(F)  .$$
Write  $F(X,T) = T^N + \sum_{i=1}^{N} f_i(X)T^{N-i}  \in k[X_1, \ldots , X_m, T]$. 

Then,
$$\frac{\partial ^j F}{\partial T^j} (x, h(x)) = 0  \; \; (j = 0, \ldots , \lfloor \frac{N}{2} \rfloor ),$$
for all $x \in \reg (h)$.    Let $\pi \colon W = \mathrm{Var}( B) \to V$ be the normalization of $V$ and consider  
$$G(T):= T^N + \sum_{i=1}^N (a_i \circ \pi) T^{N-i} \in B[T].$$
 Since 
 $$\frac{d^j G}{dT^j}(h)=0 \; 
 (j=0, \ldots , \lfloor \frac{N}{2} \rfloor)$$ 
 on $ \pi ^{-1}(\reg (h))$ these derivatives must be identically 0 on $W$.  Just suppose that $h(y_1) \ne h(y_2)$ for some $y_1, y_2$ lying over the same point $x \in V$. Letting $\ell =\lfloor \frac{N}{2} \rfloor$ we see that  
 $$(T - h(y_1))^{\ell + 1}(T-h(y_2))^{\ell + 1} \mid T^N + \sum_{i=1}^N a_i (x) T^{N-i},$$
 which is impossible. We may conclude that $h$ is constant on the fibers of $\pi$ and hence is in the weak normalization of $A$. 
\end{proof}

Note that if an element $h$ in the normalization of $A$ satisfies $\Gamma_h \subset p^{-1}(V) \cap ZZ(F)$ as above and  therefore $h \in {}^*A$, we may regard $h$ as a globally defined continuous function on $V$.  If we now let $\Gamma_h$ denote the graph of this globally defined function, then $\Gamma_h = p^{-1}(V) \cap ZZ(F).$

If instead of assuming that $h \in {}^*A$ and that $\Gamma_h \subset p^{-1}(V) \cap ZZ(F)$ in the hypotheses of the above theorem we assume that $h$ is a $k$-valued function on $V$ and that $\Gamma_h = p^{-1}(V) \cap ZZ(F)$, we again get a characterization of the elements of ${}^*A$.

 For a meromorphic function $h$ on complex analytic space $V$ let $\reg(h)$ denote the set of points where $h$ is holomorphic and let $\Gamma_h \subset \reg(h) \times \IC$ denote the graph of the function $h: \reg(h) \rightarrow \IC$. Suppose that $Y  \subset \IC^{m+1}$  is a complex analytic space   such that the projection onto the first $m$ factors $p: Y \rightarrow \IC^m$ is a surjective finite morphism.  Think of $Y$ as
 a branched cover of $\IC ^m$. 
 If the degree of the cover is $N$, then as the analysis before Proposition \ref{alg branched cover} shows, the restriction of the projection to the set of points where $\ell = \lfloor \frac{N}{2} \rfloor$ sheets come together is a homeomorphism;  let $Y_0$ be this locus.  Let $T$ denote the new coordinate on 
 $\IC^{m+1}$. We will use this notation below.

 We are ready to give our geometric interpretation in the complex analytic setting.

\begin{prop}  \label{anal branched cover} Let $V \subset \IC^m$ be  an irreducible                                               
complex analytic space, let $A$ be  the local ring of germs of holomorphic functions at a point $x \in V$, and let $h$ be an element of the normalization of $A$. Then, $h$ is in the weak normalization of $A$ if and only if there exists an open neighborhood $U$ of $x$ in $ \IC^m$, and a  complex analytic space  $Y  \subset \IC^{m+1}$   such that the projection onto the first $m$ factors $p: Y \rightarrow \IC^m$ is a surjective finite morphism and such that
$$ \Gamma_h \subset p^{-1}(V) \cap Y_0.$$
\end{prop}

\begin{proof}
Assume that  $Y  \subset \IC^{m+1}$ is an analytic space  over a neighborhood $U$ of $x$ such that the projection onto the first $m$ factors $p: Y \rightarrow \IC^m$ is a surjective finite morphism satisfying  $ \Gamma_h \subset p^{-1}(V) \cap Y_0$. (This implies that $Y$ is a hypersurface, hence given by 1 equation, which by the preparation theorem we can take to be monic in $T$.)
Consider the graph of $h$
in the product $V \times \IC\subset \IC^{m+1}$.  Now every component of the  restriction of our branched cover to $V$ must have dimension the same as
that of $V$. Since the locus of points where $\ell = \lfloor \frac{N}{2} \rfloor$  sheets of the cover come together is closed and contains the graph of  $h$, the graph of $h$ must be a component of the restriction of the branched cover to $V$, and as  $h=T$    on this component, $h$ is analytic, and the map from the graph to $V$ is a homeomorphism. Hence, we must have $h$ is
weakly subintegral over $V$.

Now assume that $h$ is in the weak normalization of $A$ so satisfies a sequence of equations as in  (\ref{RRSeqns}). Then the zero locus $Y$ of the monic equation of highest degree $2q+1$ in $T$ in condition (\ref{RRSeqns}) defines a  branched cover on some neighborhood $U$of $x$ in $\IC ^m$
of degree $2q+1$. The vanishing of these equations in $U \subset \IC^{m+1}$ is exactly the locus
of points where $q+1$ sheets of the cover come together,  hence $\Gamma_h \subset p^{-1}(V) \cap Y_0$.
\end{proof}

Note that by the discussion before Proposition \ref{alg branched cover}  the condition that  that the graph $ \Gamma_h $ is a subset of $Y_0$ implies that $h$ satisfies  a sequence of equations as in  (\ref{RRSeqns}).

Now we consider the weak subintegral closure of an ideal. We use the definition proposed by Vitulli and Leahy \cite{VL}, which in turn is based on the 
criterion of Reid, Roberts and Singh \cite{RRS}.  Their definition stands in the same relation to the definition of Reid, Roberts and Singh, as
the definition of the integral closure of an ideal does to the normalization of a ring.

Consider $I \subset A\subset B$. We say $b \in B$ is
{\it weakly subintegral over
I} provided that there exist
 $q \in \N$ and  $a_i\in{I}^i$, for  $(1\le{i}\le 2q+1$),
such that
\begin{equation}\label{LVeqns}
b^n+\sum_{i=1}^{n} {n \choose i} a_ib^{n-i}=0\quad (q+1\le n \le 2q+1).
\end{equation}

\noindent We let
$${}^*_BI = \{ b \in B \; | \; b \hskip 4pt{\rm is\hskip 4pt weakly\hskip 4pt
subintegral\hskip 4pt over\hskip 4pt }I\}.$$
We call ${}^*_BI  $ the {\it weak subintegral closure of
$I$ in $B$}.  We write ${}^*I$ instead of ${}^*_AI  $ and refer to ${}^*I$ as the {\it weak subintegral closure of $I$}. 

\bigskip

 The paper \cite{VL} contains an important link between weak normalization of a graded ring and weak subintegral closure of an ideal, which we recall. Suppose that $A \subseteq B$ are rings, $I$ is an ideal in $A$, and $b \in B$. Then, $b$ is weakly subintegral over $I^m$ if and only if the element $bt^m \in B[t]$ is weakly subintegral over the Rees ring $A[It]$ ( \cite[Lemma 3.2]{VL}).  Thus ${}^*_BI  $ is an ideal of ${}^*_BA$ (c.f.  \cite[Prop. 2.11]{VL}).  In particular, ${}^*I$ is an ideal of $A$.  Vitulli and Leahy also  show that for an ideal $I$ in a reduced ring $A$ with finitely many minimal primes and total quotient ring $Q$, we have ${}^*(A[It]) = \oplus_{n \ge 0} \, {}^*_Q (I^n)t^n$ ( \cite[Corollary 3.5]{VL}).

\section{Local Analytic and Algebraic Characterizations for Ideals}

Throughout this section for an ideal $I$ of a Noetherian ring $A$ we let
$\overline{I}$ denote the integral closure of $I$ and  
${}^*I$ denote the weak subintegral  closure of $I$ as defined in
\cite{VL}.  Similarly, we let $\overline{A}$ and ${}^*A$ denote
the normalization and weak normalization of $A$, respectively.  For a graded ring $R$, a homogeneous element $f \in R$, and a homogeneous prime ideal $\mathfrak{q} \subset R$, we let $R_{(f)}$ and $R_{(\mathfrak{q})}$ denote 
 the degree 0 parts of the graded rings obtained by localizing with respect to the homogenous multiplicative subsets $\{1,f, f^2, \ldots \}$ and $(R \setminus \mathfrak{q}) \cap R^h$, respectively, where $R^h$ denotes the set of homogeneous elements of $R$.  To avoid confusion, if $(f)$ is a homogeneous prime, we will let $\q=  (f)$ and write $R_{(\q)}$ for the degree 0 part of $R$ localized with respect to $(R \setminus \mathfrak{q}) \cap R^h$.
 
\begin{notn}  For an ideal $I$ of a Noetherian ring $A$ and an
element $a \in A$ we write
$\ord_I(a) = n$ if $a \in I^n \setminus I^{n+1}$ and $\ord_I(a)
= \infty$ if $a \in \bigcap_{n \ge 1} I^n$.  Next we define 
$$\overline{v}_I(a) = \lim_{n \to \infty} \frac{\ord_I(a^n)}{n}.$$  
The indicated limit always exists (possibly
being $\infty$; \cite[Prop. 11.1]{Mc}) and 
$\overline{v}_I$ is called the asymptotic Samuel function of
$I$.  
For a non-nilpotent ideal
$I$ we let
$$\vv(I) = \{(V_1, \m _1), \ldots , (V_r,\m_r) \}$$ 
denote the set of Rees valuation rings of
$I$ and let $\{ v_1,\ldots, v_r \}$ denote the corresponding Rees 
valuations.  Let
$\mathcal{MR}(I)$ denote the set of minimal reductions of $I$. 
For an $\N$-graded ring $R$ we let $R_+ = \oplus_{n>0} \, R_n$. 
\end{notn}
%:  lem:S+
\begin{lemma} \label{lem: S+} Let $I$ be a nonzero proper ideal in a
reduced Noetherian local ring $(A, \m, k), \\ R=A[It], T =
\overline{R}$, and  $S = {}^* R$.  
\begin{enumerate}
\item Suppose that $\q \in
\spec(T)$ contains $It$.  Then, $T_+ \subseteq \q$.
\item Suppose that $\q \in \spec(S)$ contains $It$.  Then, $S_+
\subseteq \q$.
\end{enumerate} 
\end{lemma}

\begin{proof} Let $I = (a_1, \ldots, a_{\ell})$, $B = {}^*A$, 
$C = \overline{A}$, $J = IB$, and $K=IC$.  Recall that $S =
\oplus_{n \ge 0} {}^*(J^n)t^n$ and $T =
\oplus_{n\ge 0} \overline{K^n}t^n$.  Suppose that
$bt^d \in T_d, \, d \ge 1$.  Then $b
\in \overline{K^d}.$  Thus $bt^d$ satisfies
an equation of integral dependence
$$(bt^d)^n + c_1t^d(bt^d)^{n-1} + \cdots + c_nt^{dn}=0,$$
where $c_i \in K^{di}$.  Thus $c_it^{di}$ is a $C$-linear
combination of monomials of degree $di$ in $a_1t, \dots,
a_{\ell}t$.  By assumption, each such monomial is in $\q$ and
hence so is each $c_it^{di}$. Hence each $c_it^{di} \in \q$ and 
$(bt^d)^n \in \q$ and therefore $bt^{d} \in \q$. For the second
assertion, if we start with an element $bt^d \in S_d$ then $b \in
{}^*(J^d)$ and we get a similar equation of linear dependence where each
$c_it^{di}$  is a $B$-linear combination of monomials of degree $di$ in
$a_1t, \dots, a_{\ell}t$.  Hence $bt^{d} \in \q$ as above.
\end{proof}

Next we develop some of the properties of weak subintegral closure of ideals and some analogous results in the theory of integral closure from the joint perspectives of complex analytic geometry and commutative algebra.  We  begin with a criterion for an element to be in the weak subintegral closure of an ideal that uses blow-ups. This is analogous to the condition given by Teissier and Lejeune-Jalabert \cite{LT} for an element to be in the integral closure of an ideal. We present separate results for the complex analytic and purely algebraic settings in hope of reaching a wider audience.  

Our first pair of results is modeled  on a result proved by Teissier and  LeJeune-Jalabert which links the integral closure of an ideal with the pullback of the ideal to the normalized blow-up by the ideal.

%: theorem L-J analytic
\begin{theorem} \label{L-J analytic}  Let 
 $I$ be an ideal in a local ring $\cO_{X,x}$ of an analytic space $X$. Denote the normalization of the 
blow-up of $X$ by $I$ by $NB_I(X)$ with projection map $\pi$.  Then
given $h \in \cO_{X,x}, h\in \bar I$ if and only if   $h\circ\pi\in \pi^*(I)\cO_{NB_I(X),y} \;$ for all $y\in\pi^{-1}(x)$.
\end{theorem}

\begin{proof} This immediately follows from  \cite[Theorem 2.1]{LT}  and  \cite[Proposition 1 of Section 1.3.1]{T}.
\end{proof}

We present the algebraic version of this result after first proving a lemma.  

%: lem: local char of int cl
\begin{lemma} \label{lem: local char of int cl} Let $(A, \m )$ be a Noetherian local ring, $I = (a_1, \ldots, a_{\ell})$ an ideal of $A$, $a$ an element of  
$A, R=A[It]$, and suppose $R \subseteq S$ is an integral extension of $\N $-graded rings.  Then, 
$$a \in IS_{(\q ) } \; \forall \q \in \proj(S) \mbox{ lying over } \m \Leftrightarrow a \in IS_{(a_it)} \; (i=1, \ldots , \ell).$$

\end{lemma}

\begin{proof} Suppose that $a \in IS_{(a_it)} \; (i=1, \ldots, \ell)$.  Let $\q \in \proj(S)$ be a prime ideal  lying over $\m$.  Since $S_+ \nsubseteq \q$ we must have $a_it \notin \q$ 
for some $i$ by Lemma \ref{lem: S+}.  Hence $a \in IS_{(\q )}$.  

Now assume that $a \in IS_{(\q )} \; \forall  \q \in \proj(S)$ lying over $\m$. 
 Just suppose that $a \notin IS_{(a_it)}$ for some $i$.  Then, $a \notin IS_{a_it}$ for some $i$, since $a \in R_0 \subseteq S_0$.  Thus $IS_{a_it}:_{S_{a_it}}a = (IS:a)_{a_it}$ is a proper homogeneous ideal of $S_{a_it}$.  Hence $IS:a$ is contained in some homogeneous prime $\q$ of $S$ that doesn't contain $a_it$.  Hence $S_+ \nsubseteq \q$.
Now $\q = \q \cap S_0 + \q \cap S_+$.  We may enlarge 
 $\q$ if necessary  (by replacing $\q \cap S_0$ by $\mathfrak{n}$ where $\mathfrak{n}$ is some maximal ideal of $S_0$) and assume that $\q \cap A = \m$ and $IS : a \subseteq \q$, contradicting our assumption.  Thus, $a \in IS_{(a_it)}$ for all $i$. 
\end{proof}

We are ready to present the algebraic version of Theorem \ref{L-J analytic}.

%: prop: local I
\begin{theorem} \label{prop: local I} Let $I$ be an ideal in
a reduced Noetherian local ring $(A, \m, k),  \\R=A[It], S =
\overline{R}$, and $a \in A$.  Then,
$$ a \in \overline{I} \Leftrightarrow  a \in IS_{(\q)} \; \forall \q  
\in \proj(S) \makebox{ such that } \q \cap A = \m.$$
\end{theorem}  

\begin{proof} The assertions clearly hold if either $I$ is either the zero ideal or all of $A$.  So assume   $I = (a_1, \ldots, a_{\ell})$ is a proper nonzero ideal of $A$. Consider an element
$a \in \II$.   

We first consider the case where
$A$ is an integral domain.   Let $\mathfrak{q} \in
\proj(S)$ be such that $\mathfrak{q} \cap A = \m$.  Since $\q \in
\proj(S)$ we must have $a_it \notin \q$ for some $i$ by the preceding
lemma.  Now
$IS_{a_it} = a_iS_{a_it}$ is a principal ideal in the Krull domain
$S_{a_it}$ and hence is a normal ideal. Since taking integral closure of
ideals commutes with localization, $a
\in IS_{a_it}$ and hence $a \in IS_{\q}$.  Since $a \in R$ is
homogeneous of degree 0 this means $a \in IS_{(\q)}$.

Now assume that $a \in IS_{(\q)} \; \forall \q
\in \proj(S) \makebox{ such that } \q \cap A = \m$.  By the preceding lemma,
this implies $a \in \bigcap_{i = 1}^{\ell}
IS_{(a_it)}$.  
Thus $a/a_i \in S_{(a_it)}$ for $i = 1, \ldots, \ell$.  We claim
this implies that  $\overline{v}_I(a) \ge 1$.  
 Let $v$ be any Rees
valuation of $I$ and let $v(I) = v(a_i)$.  Since $a/a_i =
bt^d/(a_it)^d$ for some $b \in S_d = \overline{I^d}t^d$ we have
$aa_i^d = a_ib$.   Hence $v(a) + dv(a_i) = v(a_i) + v(b) \ge
v(a_i) + dv(a_i) \Rightarrow v(a) \ge v(a_i) = v(I)$ and hence
$v(a)/v(I) \ge 1$.  Since this holds for every Rees valuation and
$\overline{v}_I(a) = \min\{v_j(a)/v_j(I) \}$ we must have
$\overline{v}_I(a) \ge 1$.  Hence $a \in \II$ as desired.
  
Now let $A$ be a Noetherian reduced ring with minimal primes
Min$(A) = \{ P_1, \ldots, P_u \}$.  Recall that $a \in \II
\Leftrightarrow a + P_i \in \overline{I+P_i/P_i}$ for all $i$. 
Let $Q_i = P_iA[t] \cap A[It] \; (i=1, \ldots, u)$.  Then, $S =
\overline{R} = \overline{R/Q_1} \times \cdots \times
\overline{R/Q_u} = S_1 \times \cdots \times S_u$.  Since $a \in
IS_{(\q)} \; \forall \q \in \proj(S)$  such that $ \q \cap A = \m$ if and only if 
$a + P_i \in I(S_i)_{(q)} \; \forall \q
\in \proj(S_i)$ such that $\q \cap A = \m
\; (i = 1, \ldots, u)$, the result for reduced rings follows from
the integral domain case.
\end{proof}

We point out that in the case of an Noetherian local domain, the middle portion
 of the proof of the above theorem follows from a result that
appeared in \cite[Lemma 3.4]{Hu} and is due to J. Lipman. We recall
that result now.

\begin{lemma} Let $(R, \m)$ be a (Noetherian) local domain, $I$ an ideal
of
$R$.  Then, $\overline{I} = (\cap IV ) \cap R$, where the
intersection is taken over all discrete valuation rings in the quotient
field of $R$ which contain $R$ and have center $\m$.

\end{lemma}

 Lipman's original result \cite[Proposition 1.1]{Lip}  had the additional assumption that $R$ is universally catenary and the extra conclusion that the intersection can be taken over those discrete valuation rings in the quotient field of $R$ which contain $R$, have center $\m$, and such that the transcendence degree of $R_v/\m _v$ over $R/\m$ is $\dim(R) -1$, which is the largest it can possibly be.  

In the analogs we want to replace normalization by weak normalization. We let $R(I)$ denote the Rees algebra of an ideal $I$.  Recall that a regular ideal is an ideal that contains a regular element, that is, contains a non-zero divisor. An ideal on an analytic set $X$ is regular provided it does not vanish on a component of $X$.
 
%: local anal char of *I
\begin{theorem} \label{local anal char of *I} Suppose that either $I$ is a regular ideal in a local ring $\cO_{X,x}$ of an analytic space $X$ and $h \in \cO_{X,x}$, or $I$ is an arbitrary ideal and $h$ vanishes on each component of $X$ on which $I$ vanishes. Denote the weak normalization of the
blow-up of $X$ by $I$ by $WNB_I(X)$ with projection map $\pi$.  Then, $h$ is in the weak subintegral closure ${}^*I$ of $I$  if and only if $h\circ\pi\in \pi^*(I)\cO_{WNB_I(X),y}$  for all
$y\in\pi^{-1}(x)$.
\end{theorem}

\begin{proof} Suppose $h$ is in ${}^*I$. Then $h$ satisfies a system of equations as in (\ref{LVeqns}); pull this system and $h$ up to
$WNB_I(X),y$, $y\in\pi^{-1}(x)$, then by the argument of 2.2 of \cite{VL}, the pull back of $h$ is in the ideal generated by the pullback of $I$.

Now suppose $I$ is a regular ideal, $h\circ\pi\in \pi^*(I)\cO_{WNB_I(X),y}$ for all
$y\in\pi^{-1}(x)$. By Teissier's result this implies that
$h\in\bar I$. Consider $\Projan(R(I+(h)))$ and $\Projan(R(I))$.  Since $h$ is integrally dependent on $I$,
it's not hard to see that locally $\Projan(R(I+(h)))$ is just the closure of the graph of $h/p$ where $p$ is a local generator of the pullback of $I$ on
$\Projan(R(I))$. More is true. By hypothesis, on $WNB_I(X)$, the quotient $h/p$ is smooth locally, and the map to the graph of $h/p$ must be finite surjective and
1-1. Now on the graph $h/p$ is holomorphic, so the inclusion of local rings of $\Projan(R(I))$ at $y$ into the local ring of $\Projan(R(I+(h)))$ at $y, h(y)/p(y)$ is a
subintegral extension.

Now we show that $R(I+(h))$ is a subintegral extension of $R(I)$.
We can think of the prime spectrums of these two rings as embedded in the product of $X$ and a
suitable affine space, $\IC^N$, the affine space for
$\spec(R(I+(h)))$ being one dimension bigger. We get an induced map of prime spectrums by projection from $\IC^{N+1}$ to $\IC^N$. This induced map of prime spectrums is a bijection.
We can see this in two steps.

Step 1. If we work on the sets where
${z_i\ne 0}$, where $\{z_1,\dots,z_N\}$ are coordinates on $\IC^N$ and
$\spec (R(I)) \subset X\times \IC^N$, then these subsets are locally products
of the corresponding open affines on the projective spaces
and the induced map respects the product structure.

Step 2.  From Step 1 we know that if we remove $X\times \{0\}$ from the prime spectrums, then we have a bijection; but we already have a bijection
on $X\times \{0\}$ so we have a bijection between $\spec(R(I+(h)))$ and $\spec(R(I))$.
Since we have a bijection of prime spectrums and the condition is a local one, we're done.
Now  \cite[Theorem 3.5]{VL} implies $h \in {}^* I$ in this case.

If $I$ is not a regular ideal, and $X'$ is the union of those components  of $X$ on which $I$ is not zero, then the above argument shows that $h$ is in the weak subintegral closure of the ideal $I$ induces on $X'$. Consider the set of equations satisfied by $h$ on $X'$. Pull these back to $X$. Now the right hand side of the equations may not a priori be zero--instead it may be some functions $g_i$ which vanish on $X'$. However, by hypothesis, $h$ and $I$ vanish on the components of $X$ off $X'$, so the $g_i$ must as well, since the left hand side of the equations vanish on these components. Hence $g_i$ vanish identically on $X$ and $h$ is in the weak subintegral closure of $I$.
\end{proof}

From part of the proof of Theorem \ref{local anal char of *I}, we extract the following proposition.

\begin{proposition}  \label{proposition homeo to subintegral} Suppose $I$ is a regular ideal in a local ring $\cO_{X,x}$ of an analytic space $X$ and $h \in \cO_{X,x}$, and  the inclusion of local rings of $\Projan(R(I))$ at $y$ into the local ring of $\Projan(R(I+(h)))$  at $y, h(y)/p(y)$ is a
subintegral extension, for all $y\in\pi^{-1}(x)$. Then $R(I+(h))$ is a subintegral extension of $R(I)$.
\end{proposition}

\begin{proof}  Since $I$ is regular, it follows that there is a 1-1 correspondence between components of
$\Projan(R(I))$, and $\Specan(R(I))$. As in the proof of Theorem \ref{local anal char of *I}, the map between the prime spectrums is induced by projection from a suitable $\IC^{N+1}$ onto $\IC^N$, and the hypothesis implies that this map is finite, as $h$ is in the integral closure of $I$. Now the result follows from Steps 1 and 2 of the proof of Theorem \ref{local anal char of *I}.
\end{proof}

We now present the algebraic version.

%: theorem  local alg char of *I
\begin{theorem} \label{local alg char of *I} Let $I$ be an ideal in a Noetherian
reduced ring   $(A, \m, k)$, $ R=A[It], S = {}^*R$,
and $a \in A$.  Then,
$$ a \in {}^*I \Leftrightarrow  a \in IS_{(\q)} \; \forall \q
\in \proj(S) \makebox{ such that } \q \cap A = \m.$$

\end{theorem}

\begin{proof} 
  Choose  generators $a_1, \ldots, a_{\ell}$ of $I$.
  %as in \cite[2.12 Lemme]{LT}.   
   First suppose that $a \in
{}^*I$.  Let $\q \in \proj(S)$ be such that $\q \cap A = \m$.  Then
$a_it \notin \q$ for some $i$ by Lemma \ref{lem: S+}.  Then, $IS_{a_it}
= a_iS_{a_it}$ is a principal regular ideal in the weakly normal ring
$S_{a_it}$ and hence $a_iS_{a_it}$ is weakly normal by \cite[Remark 2.2
]{VL}.   Since $a \in {}^*I \subseteq {}^*(IS_{a_it})$ we must have $a
\in IS_{a_it}$.  By a degree comparison, $a \in IS_{( a_it )}$.  Since
$a_it \notin \q$ we have $a \in IS_{(\q)}$.
 
Now assume that $a \in IS_{(\q)} \; \forall \q
\in \proj(S) \makebox{ such that } \q \cap A = \m$.  In particular, $a \in \II$ by Proposition \ref{prop: local I} and $S \subseteq T: = S[at]$ is a finite integral extension.   We also know that $a \in \cap IS_{(a_it)} = \cap a_iS_{(a_it)}$ 
by Lemma \ref{lem: local char of int cl}.    Hence $at/(a_it) \in S_{(a_it)} \; (i = 1, \ldots \ell)$ and 
 $at \in S_{(a_it)} \; (i=1, \ldots, \ell)$.    We just demonstrated that 
$S_{a_it} = T_{a_it} \; (i=1, \ldots, \ell)$.  Hence $S_{\q} = T_{\q}$ for every $\q \in \proj(S)$ by Lemma \ref{lem: S+}.   Just suppose that $S \subsetneq T$.  Then $S:T$ is a nonzero homogeneous radical ideal  of $T$ and is contained in $S$, since $S$ is weakly normal and hence seminormal in $T$.  Let $\q $ be a minimal overprime of $S:T$ in $S$.  Then $\q$ is homogeneous and $S_{\q} : T_{\q} = \q S_{\q} $ as an ideal in $S_{\q}$.  By the above observations, $S_+ \subseteq \q$ and $\q = \q \cap S_0 + S_+$.  Notice that $S_0 = T_0$.  Let $Q$ be any overprime of $\q$ in $T$.  We have $T_+ \subseteq Q$ by Lemma \ref{lem: S+}.  Hence $Q$ is homogeneous and $Q = Q \cap T_0 \, + T_+ = \q \cap S_0 + T_+$, that is, there is a unique prime ideal in $T$ lying over $\q$.  Let $S' = S_{\q}, T' = T_{\q} = T_Q$ and $\q ' = \q S_{\q}, Q' = QT_Q$.  Then $(S', \q ') \subseteq (T', Q')$ is a finite integral extension of reduced local rings (not necessarily Noetherian) such that $S':T' = \q ' = Q'$.  Furthermore, $S_0/(\q \cap S_0) \cong T/Q$, which implies $S' = T'$, a contradiction.
\end{proof}

We offer some corollaries to the these results.

%: lem homeo to subintegral
\begin{lemma} \label{lem homeo to subintegral} Suppose that $(A, \m, k)$ is the local ring of an algebraic variety over an algebraically closed field of 
characteristic 0 , $I$ is a 0-dimensional ideal in $A$, and $h \in A$.  Finally, let $R = A[It]$ and $S = A[It, ht]$.   If the induced map $\proj(S) \to \proj(R)$ is a homeomorphism, then $R \subset S$ is a weakly subintegral extension, that is, $h \in {}^*I$. 
\end{lemma}

\begin{proof}
If $\dim(A) = 0$ there is nothing to prove.  So assume the $\dim(A) > 0$ and choose generators $g_1, \ldots , g_{\ell}$ of $I$.  Let $T = {}^*R$.  By Lemma \ref{lem: local char of int cl}  and Theorem \ref{local alg char of *I} it suffices to check that $ht \in IT_{(g_it)}$ for $i = 1, \ldots , \ell$.  Since the blow-ups are homeomorphic and we are working over an algebraically closed field of characteristic 0, we know that $R_{(g_it)} \subset S_{(g_it)}$ is  a weakly subintegral extension for $i = 1, \ldots , \ell$ and hence $S_{(g_it)} \subset {}^*(R_{(g_it)}) \subset T_{(g_it)}$.   We also know that $h= \frac{ht}{g_it} g_i \in IS_{(g_it)} \subset IT_{(g_it)}$. This finishes the proof.
\end{proof}

In the analytic case we have the following consequence.

\begin{corollary} Suppose $\mathcal{I}$ is a coherent sheaf of ideals on an analytic space $X$. Form the ideal sheaf  $\; {}^*\mathcal{I}$ by taking the weak subintegral closure of each stalk.
Then the result is a coherent sheaf.
\end{corollary}

\begin{proof} Pullback $\mathcal{I}$  to $WNB_{\mathcal{I}}(X)$; there it generates a coherent sheaf, 
push the sheaf down to $X$ and intersect with $O_X$, the result is ${}^*\mathcal{I}$ by Theorem \ref{local anal char of *I}.
\end{proof}

\section{The Ideal $I_{>}$ and Connections with Reductions}

Recall that if $J \subset I$ are finitely generated ideals such that $\overline{J}=\overline{I}$ then,  $J$ is a {\it reduction} of $I$. 
 
%: notation I greater than
\begin{notn}
For an ideal $I$ in a Noetherian ring $A$ we let
$$I_> =  \{ a \in A \mid \overline{v}_I(a) > 1 \}.$$  
In general, $I_>$ is an ideal of $A$ and 
a subideal of $\overline{I}$.
\end{notn}
 Let $I$ be a regular  ideal in a Noetherian ring $A$ (i.e., $I$ contains a non-zero divisor) and $a \in A$.  The asymptotic Samuel function $\overline{v}_I$ is determined by the Rees valuations of $I$.  Namely,
$$\overline{v}_I(a) = \min_j \left\{ \frac{v_j(a)}{v_j(I)}
\right\},$$ where $v_j(I) = \min \{ v_j(b) \mid b \in I \} $ (see \cite[Lemma 10.1.5]{H-S}).
Recall that $\overline{v}_I = \overline{v}_J$ whenever $\overline{J}  = \overline{I}$ (see \cite[Cor. 11.9]{Mc}).  This immediately implies that $J_> = I_>$ whenever $\overline{J}  = \overline{I}$.
\medskip

We now prove a quick lemma.

\begin{lemma}  Let $I$ be a regular ideal in a Noetherian ring $A$. 
Then,
$$I_> = \bigcap_i  \m_iIV_i \cap A.$$
In particular, $I_>$ is an integrally closed ideal.
\begin{proof}  Let $a \in A$.  Notice that $a \in I_>$ if and
only if $v_j(a)  > v_j(I)$ for all Rees valuations $v_j$ of $I$. 
Since $(V_j,\m_j)$ is a discrete rank one valuation ring the
latter is true if and only if $a \in \m_jIV_j$ for all
$(V_j,\m_j) \in \vv(I)$ .  
\end{proof}

\end{lemma}

We would like to explore what this lemma implies about monomial ideals.  First we need a lemma.

\begin{lemma}  Let $I$ be a nonzero monomial ideal in a polynomial ring over a field. Then, $I_>$ is again a monomial ideal.
\end{lemma}

\begin{proof}  Recall that the Rees valuations of $I$ are monomial valuations and correspond to the bounded facets of the Newton polyhedron of $I$.  A polynomial $f$ is in $I_>$ if and only if $v(f) > v(I)$ for all $v \in \mathcal{RV} (I)$.   Now $v(f) = \inf\{v(\mu) \}$, where the infimum is taken over all monomials occurring in $f$. We may deduce that $f \in I_>$ if and only if every monomial occurring in $f$ is in $I_>$ and, hence, $I_>$ is a monomial ideal.  
\end{proof}

Let   $I \subset k[\xvec{n}]$ be a monomial ideal in a polynomial ring over a field $k$.  Then, $I_>$  is generated by all monomials whose exponent vectors do not lie on any bounded facet of the Newton polyhedron of $I$.  For example, if $I = (x^2, y^2) \subset k[x,y]$ and $\m = (x, y)$ then $I_> = \m ^3$.

It is well known that the integral closure of an ideal $I$ of a Noetherian ring $A$ consists of all elements $a \in A$ such that $\overline{v}_I(a)  \ge 1$ (e.g. see \cite{Mc} Chapter XI or \cite{H-S} Section 10.1).  One may ask is something similar holds for weak subintegral closure.   In 1999 at a Route 81 (New York) Conference,  D. Lantz conjectured that if $I$ is an $\m$-primary ideal in  a 2-dimensional regular local ring $(A, \m)$, then ${}^*I$ contains all elements $a \in A$ such that $\overline{v}_I(a)  > 1$.  In \cite{V} the second author proved that if $I$ is a monomial ideal in a polynomial ring $k[\xvec{n}]$ that is primary to the ideal $(\xvec{n})$ then ${}^*I$ contains all elements $a \in A$ such that $\overline{v}_I(a)  > 1$.  We now prove Lantz's conjecture in general. 

%: prop:lantz
 \begin{proposition}\label{prop:lantz}  Let $I$ be an ideal of a
Noetherian  ring $A$.  Then, $I_> \subseteq {}^*I$.
\end{proposition}
\begin{proof}  Suppose that $a \in I_>$.  Then

$$ \lim_{n \rightarrow \infty} \frac{\ord_I(a^n)}{n} = 1 + 2\epsilon$$
 for some
positive $\epsilon$.  In particular there is a positive integer $q$ such that $\frac{\ord_I(a^n)}{n} > 1 + \epsilon$ for all $n > q$.  Thus
$\ord_I(a^n) > n$, and hence $a^n\in I^{n+1}$,  for all $n > q$. 

Now we can construct equations showing that $a$ is weakly subintegral over $I$ as follows: Define $a_0 = a_1 = \cdots = a_{q} = 0$ and $a_{q+1} = -a^{q+1}$.  Define $a_{q+i}$ recursively for $i = 2,  \ldots , q+1$ so that
the sequence of equations we numbered as (1) is satisfied for $a$. 
\end{proof}

\noindent\textbf{Observation} It is well known that $a \in \overline{I}$ if and only if there exists an integer $k$ 
%and $N_0$ depending on $r$ 
such that $a^n\in I^{n-k}$ for  all $n>k$ (see, for example, \cite {H-S}, Cor. 6.8.11).  The proof of \ref{prop:lantz} shows that $a \in I_>$ implies we can take $k=-1$.   The argument given in the proof of \ref{prop:lantz} shows that if there exists a nonnegative integer $k$ such that $a^n \in I^n$ for all $n > k$ then $a \in \II$.

In the complex analytic case, an alternate proof of \ref{prop:lantz} can be given using Theorem  \ref{local anal char of *I}, and the connection between $\overline{v}_I$ and the Rees valuations of $I$ as follows. Pullback $a$ and $I$ to $WNB_{ I}(X)$, and consider $a/p$, $p$  a local generator of the pullback of $I$ on
blow-up by $I$. Then  $\overline{v}_I(a)>1$ implies that the quotient $a/p$ is zero when $p$ is 
zero, hence continuous on $WNB_{ I}(X)$, hence analytic from seminormality. This implies that the pullback of $a$ to $WNB_{ I}(X)$ is  in the pullback of $I$ and the result follows from \ref{local anal char of *I}.

\medskip

One immediate consequence of generalization of  Lantz's conjecture is the following. 

%: cor:lantz
\begin{cor}  \label{cor:lantz} Let $I$ be an ideal of a Noetherian  ring $A$. 
Then, $$I_> \subseteq \bigcap_{J \in \mathcal{MR}(I)} {}^*J.$$

\end{cor}
 \begin{proof}  Observe that if $J$ is any reduction of $I$ then 
$\overline{v}_J = \overline{v}_I$ \cite[Corollary 11.9]{Mc} and
hence $J_> = I_>$.  The assertion immediately follows from
Proposition
\ref{prop:lantz}.
\end{proof}

\medskip

We further explore the connection between $I_>$ and minimal reductions in the following theorem.

%: thm:weaknor of min red
\begin{theorem}  \label{thm:weaknor of min red} Let $(A, \mathfrak{m},k)$ be a Noetherian local
ring  such that $k$ is algebraically closed of
characteristic 0.  Suppose that $I$ is an
$\mathfrak{m}$-primary ideal.   If $J$ is any minimal reduction
of $I$ then $J + I_> = {}^*J$.

\end{theorem}

\begin{proof}  We may reduce to the case where $I$ is integrally closed since a minimal reduction $J$ of $I$ is a minimal reduction of $\overline{I}$ and $I_> = \overline{I}_>$.  So assume that $I$ is integrally closed.

Let $d$ denote the dimension of $A$.  Recall that $I_>$ is a subideal of $I$ and is
contained in ${}^*J$ by Proposition \ref{prop:lantz}.
  Let $h \in {}^*J$ and choose generators $g_1, \ldots, g_d$ for
$J$ (since $k$ is infinite all minimal reductions are
$d$-generated as in \cite[Prop. 8.3.7]{H-S}).  Let $K = (J,h)$.
	Consider the Rees algebras
$$R := A[Jt] = A[g_1t, \ldots, g_dt] \subseteq
S:= A[Kt]=A[g_1t,
\ldots, g_dt, ht],$$
and the associated fiber cones
$$\mathcal{N}_{ J} = R/ \m R  \to \mathcal{N}_{ K}= S/ \m S.$$
  Since $ht \in {}^*R$ we know that the map
$$\spec ( S) \rightarrow \spec (R)$$
is a homeomorphism of the underlying topological spaces.
 Consider the induced homomorphism of  $\N$-graded rings
$$ \mathcal{N}_{ J} \to ( \mathcal{N}_{K})_{\mathrm{red}}.$$
  Let  $z_i$
denote the image of $g_it$ in $\mathcal{N}_{J}$ for $i = 1 , \ldots , d$. Notice that $ \mathcal{N}_{ J} = k[z_1,\ldots, z_d]$, where the $z_i$ are algebraically independent over $k$ (e.g., see  \cite{H-S} Corollary 8.3.5).  In particular, $\m R$ is a prime ideal of $R$ of height 1 (recall that $I$, and hence $J$, has finite colength).  Let $\q$ denote the unique prime ideal of $S$ lying over $\m R$. We must have $\sqrt{\m S} = \q, \; ( \mathcal{N}_{K})_{\mathrm{red}} =  S/ \q$,  and $\mathcal{N}_{J} \subseteq ( \mathcal{N}_{ K})_{\mathrm{red}} $. Viewing $ S/ \q$ as an $R / \m R$-algebra we have
$\mathcal{N}_{K}/ \q = \mathcal{N}_{J}[z]$ where $z$ denotes the image
of $ht$ in $( \mathcal{N}_{K})_{\mathrm{red}}$.  More is true.  
First note that this is again a weakly
subintegral extension.   In general, if $R \subset S$ is a weakly subintegral extension and $\q \in \spec(S)$, then $R/(\q \cap R) \subset S/ \q$ is again a weakly subintegral extension.

Thus $S/ \q$ is an integral
domain and is of the form $S/ \q = k[z_1,
\ldots,z_{d+1}]/(F)$, where $z_1, \ldots, z_{d+1}$ are
algebraically independent over $k$, $F$ is a monic polynomial in
$z_{d+1}$ and is homogeneous as a polynomial in $z_1, \ldots,
z_{d+1}$.  Let $m = \deg(F)$.  Since $k = \overline{k}$ has
characteristic 0, the map $\vvar(S/ \q) \to
\vvar(R/ \m R)$ is generically $m$-to-one.  Since this map
is a bijection of the underlying point sets we must have $m =
1$.  Hence the polynomial $F = z_{d+1} - (\overline{a_1}z_1 + \cdots + \overline{a_d}z_d)$, where for an element $a \in A$ we let $\overline{a}$ denote its image in $A/\m$.  Thus 
$z - (\overline{a_1}z_1 + \cdots +\overline{ a_d}z_d)  \in  \q $, and hence, $[h
- (a_1g_1 + \cdots + a_dg_d)]^n  \in \m K^n$ for some
positive integer $n$.   Let $v_j \in \mathcal{RV}(I)$.   We may conclude 
\begin{eqnarray*}
nv_j(h - (a_1g_1 + \cdots + a_dg_d))
&=& v_j([h -(a_1g_1 + \cdots + a_dg_d)]^n )  \\
&\ge & v_j(\m K^n)   \\
& \ge & (n+1)   \\
\end{eqnarray*}
and hence
 $v_j(h - (a_1g_1 + \cdots + a_dg_d))
\ge \frac{n+1}{n}.$ Since
this is true for every $v_j \in \mathcal{RV}(I)$ we must have
$\overline{v}_I(h - (a_1g_1 + \cdots + a_dg_d)) > 1$ and hence $h
- (a_1g_1 + \cdots + a_dg_d) \in I_>$.  Therefore $h \in J + I_>$.
\end{proof}

 This result also holds for 0-dimensional ideals in an arbitrary Noetherian ring as we now show.

\begin{cor}  Let $A$ be a Noetherian ring, $\m$ a maximal ideal in $A$ such that $A/ \m$ is  algebraically closed of characteristic 0 and suppose that $I$ is an $\m$-primary ideal in $A$.  If $J$ is any minimal reduction of $I$, then $J + I_> = {}^*J$.
\end{cor}

\begin{proof}  Since $J \subset I$ are $\m$-primary, we see that $J_{\m}$ is a minimal reduction of $I_{\m}$.  By Theorem \ref{thm:weaknor of min red} we may conclude that 
$J_{\m} + (I_{\m})_> = {}^*(J_{\m})$. 
One checks that $(I_{\m})_> = (I_>)_{\m}$ and ${}^*(J_{\m}) = ({}^*J)_{\m}$.  To check the former note that $\mathrm{ord}_{I_{\m}}(a/1) =\mathrm{ord}_I(a)$ since $I$ is $\m$-primary, and hence $\overline{v}_{I_{\m}}(a/1) = \overline{v}_I(a)$,  for all $a \in A$.  To check the latter one can use \cite[Lemma 3.2]{VL} and the fact that for rings weak normalization and localization commute \cite[Cor. to Prop. 2]{Y}.  Thus $(J + I_>)_{\m} = ({}^*J)_{\m}$.  Since both ideals are $\m$-primary we may conclude that $J + I_> = {}^*J$.
\end{proof}

The next lemma and the subsequent proposition are in the spirit of the ``Integral Nakayama's
Lemma," which first appeared in the work of LeJeune-Jalabert--Teissier \cite{LT} and
was generalized to the module setting by Gaffney \cite[Proposition 1.5]{G} .   

 We recall the algebraic
version now and then prove a lemma before presenting the result.

\begin{lemma}  Suppose that $(A, \m)$ is either a Noetherian ring with Jacobson radical $\m$ or a polynomial ring over a field with $\m$ the ideal generated by the indeterminates.  Let $I_1 \subseteq I_2$ and $J \subseteq \m$ be
ideals in $A$ and in the polynomial case, assume $I_1$ and $I_2$ are monomial ideals.  If $I_1 + JI_2$ is
a reduction of $I_2$, then $I_1$ is a reduction of $I_2$.
\end{lemma}

\begin{proof}  In case, $(A, \m)$ is a Noetherian ring with Jacobson radical $\m$, this follows immediately from from \cite[Lemma 8.1.8]{H-S}.  The proof also works for monomials ideals in a polynomial ring over a field by the monomial version of Nakayama's lemma.
\end{proof}

As we alluded to previously, the first author proved a version of the above lemma for submodules of a free module of finite rank over the local ring of a complex analytic space (see \cite[Proposition 1.5]{G} ).
\begin{lemma}  Let $J \subseteq I$ be ideals in a Noetherian ring
$A$.  Then, 
$$\overline{J} = \overline{I}  \Leftrightarrow JV_i =
IV_i \; \; \forall \,  (V_i, \m_i) \in \vv(J).$$

\end{lemma}

\begin{proof}  If $\bar{J}=\bar{I}$ then $\vv(J)=\vv(I)=\vv(\bar{I})$ and hence $JV_i =
IV_i \; \; \forall \,  (V_i, \m_i) \in \vv(J).$  

Now suppose that $JV_i = IV_i \; \; \forall \,  (V_i, \m_i) \in \vv(J).$    Then
\begin{eqnarray*}
\overline{J} &=& \bigcap_{V_i \in \vv(J)} JV_i \cap A \\
 & =& \bigcap_{V_i \in \vv(J)} IV_i \; \cap A .
\end{eqnarray*}
Now $\cap_{V_i \in \vv(J)} IV_i \; \cap A$ is an integrally closed ideal containing $I$, which implies
$$\overline{I} \subseteq \bigcap_{V_i \in \vv(J)} IV_i \cap A = \overline{J},$$
and hence $\overline{J} = \overline{I}$.
\end{proof}

We now present a helpful result that uses the ideal $I_>$ to find reductions.  A similar result for modules that uses the module $M^{\dagger}$ of elements strictly dependent on $M$ in place of $I_>$ was proven by Gaffney and Kleiman  \cite[Prop. 3.2]{GK}.

\begin{prop}\label{prop:red}  Let $J \subseteq I$ be ideals in a Noetherian
ring.  If $J + (I_> \cap I)$ is a reduction of $I$, then $J$
is a reduction of $I$.
\end{prop}
 
\begin{proof}  Suppose that $K := J + (I_> \cap I)$ is a reduction of $I$.  Then,  $\vv(K)= \vv(I)$ and $K_> = I_>$ since  $\overline{v}_{K} = \overline{v}_I$.  Let $(V,\m) \in \vv(K )$.  
Then, $$ K V = IV$$ by assumption.  Since
 \begin{eqnarray*}
K V & = & JV +(I_> \cap I)V  \\
   & \subset &  JV +I_>V \\
&=& JV + \m IV,
 \end{eqnarray*}
 we must have $JV +
\m IV = IV$.  Hence $JV = IV$ by Nakayama's Lemma.  Since this holds for every Rees valuation ring of $K$, we may conclude that $K$ is a reduction of $I$ by the preceding lemma.
\end{proof}

Although we do not develop the properties of strict dependence here, we remark that for $\m$-primary ideals, the notion of $h$ being strictly dependent on an ideal $I$ in a local ring is the same as $h\in I_>$; if $I$ is not $\m$-primary, then the condition that $h \in I_>$ is more stringent. Strict dependence is a condition that holds pointwise, whereas in general, whether or not  $h\in I_>$ depends on the behavior of $h$ along the images of the components of the exceptional divisor of the blow-up by $I$.

 Given two reductions of the same ideal we know that the weak subintegral closures
both contain all those elements $h$ with $\overline{v}_I (h)>1$ by Proposition \ref{prop:lantz}. It is interesting to ask that if we fix an integrally closed 
ideal $I$, when is it the case that the intersection of the weak subintegral
closures of all reductions is exactly those elements with $\overline{v}_I (h)>1$? The answer to this question  gives another characterization
of the elements with $\overline{v}_I(h)>1$. These elements play an important role in equisingularity theory.

Before we present two examples which show the variety of phenomena that can occur we offer an observation about $m$ primary ideals and their reductions.

\smallskip

\textbf{Observation.} If $I$ is $\m$-primary ideal in the local ring
$(A, \mathfrak{m},k)$
 of dimension $d$, then $I/(I_> \cap I)$ is a $k$-vector space and
$\dim_k(I/I_> \cap I) \ge d$.  This follows from the previous proposition as we shall now see.  If the images of $a_1, \ldots, a_s$ form a $k$-basis for $I/(I_> \cap I)$ and $J = (a_1 , \ldots , a_s)$ then $J + (I_> \cap I)$ is a reduction of $I$ and hence so is $J$.  Consequently we must have $s \ge d$.  Hence the generators of any reduction $J$ must contain $d$ independent elements in $I/(I_> \cap I)$. The same statement and argument hold when $A$ is a polynomial ring, $\m$ is the ideal generated by the indeterminates, and $I$ is an $\m$-primary monomial ideal.

\begin{ex}
Let $I=(x^2,xy^2,y^3) \subset \IC [x, y]$. Then $I=\bar I$.  We claim that ${}^*J = I$ for every minimal reduction $J$ of $I$.

One can see this as follows.  First observe that the core of $I$, that is, the intersection of all reductions of $I$, is $(x^2, y^3)^2: I = (x^3, x^2y, xy^3, y^4) =: K$ by \cite[Theorem 2.3]{PUV}.  Notice that $I$ has one
 Rees valuation, namely, the monomial valuation determined by $v(x^ay^b)= 3a + 2b$.  Thus $I_> = (x^3, x^2y, xy^2, y^4) \supset K$ and $\dim_k(I/I_>) = 2$.  Suppose that $J = (f, g)$ is a minimal reduction of $I$. Then  $\dim_k((J+I_>)/I_>) =2$  by the observation. So, ${}^*J=J+I_>=I$, hence  the intersection of the ideals ${}^*J$ over all minimal reductions $J$ of $I$ is all of $I$.
\end{ex}

\begin{ex}
Let $I=(x^2,xy,y^2) = \bar I \subset \IC [x, y]$,  $J = (x^2, y^2),$ and $\m = (x,y)$.   Notice that $J$ is a minimal reduction of $I$ and is weakly subintegrally closed by \cite[Theorem 4.11]{RV}. So the multiplicity of $I_{\m}$ is the colength of $J_{\m}$ is $4$.
Again, $I$ has one Rees valuation, which is the monomial valuation determined by $v(x) = v(y) = 1$, and $v(I) = 2$.  In this case, $I_> = (x^3, x^2y, xy^2, y^3) = (x^2, y^2)^2:I = \mathrm{core} (I)$ by \cite[Theorem 2.3]{PUV}.

Let $J_a=(x^2+axy,y^2)$, $J_b=(x^2,y^2+bxy)$, where $a, b \in k^*$. Then $J_a$ and $J_b$ are also reductions of $I$, because locally this is true as follows.  Their colengths remain $4$ so their multiplicities when localized at $\m$ are the same as that of $I_{\m}$. Then ${}^*J_a\cap{}^*J_b=I_> + (bx^2+abxy+ay^2)$, so the intersection of ${}^*J$ over all reductions $J$ of $I$ is just $I_> $.
\end{ex}

Both phenomena are accounted for by the following ideas.

%:  cor: *J = I
\begin{cor} \label{ cor: *J = I} Let $(A, \mathfrak{m},k)$, where  $k$ is an algebraically closed field and either $(A, \m)$ is a local
ring of dimension $d$ or $A = k[x_1 , \ldots , x_d]$ and $\m = (x_1, \ldots x_d)$.  Suppose that $I = \overline{I}$ is an
$\mathfrak{m}$-primary ideal and in the polynomial case assume that $I$ is monomial.  
 \begin{enumerate}
  \item If $\dim _k(I/I_>) = d$, then ${}^*J =
I$ for every reduction $J$ of $I$.
  \item If $\dim _k(I/I_>) >  d$, then $\bigcap_{J \in \mathcal{MR}(I)} {}^*J= I_>$.
 
\end{enumerate}
\end{cor}

\begin{proof}  First assume that $\dim _k(I/I_>) = d$ and let $J \in \mathcal{MR}(I)$.  
Then,  $J/(J \cap I_>) = I/I_>$ and so, we must have $J + I_> = I$. Since $J + I_> \subseteq
{}^*J$ we also have ${}^*J = I$.

 Now assume that $\dim _k(I/I_>)=D >  d$.  Choose 
$g_1, \ldots, g_D$  in $I$ whose images form a $k$-basis for $I/I_>$.  The set of
minimal reductions of $(g_1, \ldots, g_D)$ can be identified with a dense 
Zariski-open subset of the space of $d$-planes in $I/I_>$, which we identify with
affine $D$-space.  Intersecting over all minimal reductions $J$ of $(g_1, \ldots,
g_D)$ we get   
$ \cap \,  {}^*J/ I_> =  \cap (J + I_> )/I_> $
is the zero subspace.  Hence
the intersection of the ideals $J + I_>$ over all minimal reductions of   $(g_1,
\ldots, g_D)$ is
$I_>$. Since every minimal reduction of $(g_1, \ldots, g_D)$ is a minimal reduction
of $I$ the result follows.
\end{proof}

\section{A Valuative Criterion}
In this final section we  develop a valuative theory for weak subintegral closure. This will be done by introducing a new closure operation, which we call relative weak closure, and giving another characterization of an element being in the weak subintegral closure of an ideal using this new idea.  In turn our criterion depends on a valuative criterion for the integral closure of an ideal that is well known for complex analytic spaces \cite[2.1 Theor\`eme]{LT} and is proven for algebraic varieties below.

%: rel closure anal
\begin{defn} \label{rel closure anal} Let $\cO_{X,x}$ be the local ring of a complex analytic space $X$. Suppose that $M\subset N\subset F$ are submodules of a free $\cO_{X,x}$-module $F$ of rank $r$. Then an element $h \in F$ is in the relative closure of $M$, denoted ${\overline {M_N}}$, if for all curves $\phi \colon (\IC,0) \to (X,x), \;   \phi^*(h) \in \phi^*(M) + \m_1\phi^*(N)$, where we are identifying $ \phi^*(M)$ and  $\phi^*(N)$ with their images in $\phi^*(F)$ and letting $\m_1$ denote the unique maximal ideal of $\cO_{\IC, 0}$.
\end{defn}

In the algebraic setting we define the relative closure in much the same fashion.  Throughout this section $k[[z]]$ denotes a formal power series ring in one variable over an algebraically closed field $k$.

%: rel closure alg
\begin{defn}  \label{rel closure alg} Let $(A, \m)$ be a reduced Noetherian local ring, essentially of finite type over an algebraically closed field $k$.   Suppose $M\subset N\subset A^r$ are submodules of the free $A$-module $A^r$. Then an element $h \in A^r$ is in the relative closure of $M$, denoted ${\overline {M_N}}$, provided that for all local homomorphisms of $k$-algebras  $\rho \colon A \to k[[z]]$, $\rho^{(r)}(h) \in \rho^{(r)}(M)k[[z]] + z\rho^{(r)}(N)k[[z]]$, where we  denote by $\rho^{(r)}$ the induced map  $\rho^{(r)} \colon A^r \to k[[z]]^r$ and identify $\rho^{(r)}(M)$ and $\rho^{(r)}(N)$ with their images in $\rho^{(r)}(A^r)$ . 
\end{defn}
 To make use of the relative closure operation in the algebraic case, we need an algebraic version of the valuative criterion for the integral closure of an ideal for complex analytic spaces.  We first prove a lemma and then present the analogous result.  One can establish these results over $\IC$ by citing the complex analytic results and using GAGA \cite{Se} to equate the completions of the local rings of the complex algebraic variety and associated complex analytic space.  We can avoid reference to the complex-analytic result by referring to a theorem of B\"{o}ger \cite[Satz 2]{B} on curve-equivalent ideals, which generalized an earlier result of Scheja \cite[Satz 2]{Sch}, as follows.

\begin{lemma} \label{Boger} Let $(A, \m, k)$ be a Noetherian local domain, essentially of finite type over the algebraically closed field $k$, and suppose that $A$ is normal.  If $a,b \in A$ but $a \notin bA$, then there exists a local homomorphism of $k$-algebras $\rho \colon A \to k[[z]]$ into a formal power series ring over $k$ such that $\rho(a) \notin \rho(b)k[[z]]$.
\end{lemma}
 \begin{proof} Let $ \mathfrak{b} = bA$, and $\mathfrak{a} = (b,a)A$.  Since $A$ is a normal domain, $\mathfrak{b}$ is an integrally closed ideal.  Since $a \notin \mathfrak{b}$, the subideal $\mathfrak{b}$ is not a reduction of $\mathfrak{a}$.  Hence by the theorem of B\"{o}ger \cite[Satz 2]{B}   there exists a dimension one prime ideal $\p$ of $A$ such that  $e((\fb + \p)/\p) \ne e((\fa + \p)/\p)$, where $e(\; )$ denotes multiplicity.  Hence 
 by the theorem of Rees \cite{R}, $(\fb + \p)/\p$ is not a reduction of $(\fa + \p)/\p$.  Now $S := A/\p$ is a 1-dimensional local domain, essentially of finite type over $k$.  Let $M$ denote the unique maximal ideal of $S$. The normalization $T$ of $S$ is a 1-dimensional semi-local domain, essentially of finite type over $k$. Letting $\alpha$ and $\beta$ denote the images of $a$ and $b$, respectively, in $S \subset T$, we may deduce that $\alpha \notin \beta T$.  Hence 
 $\alpha \notin \beta \widehat{T}$, where $\widehat{T}$ denotes the $M$-adic completion of $T$.    Thus for some minimal prime $\q$ of $\widehat{T}$ we must have $\alpha + \q \notin ( \beta + \q)(T/\q)$.  Now 
 $\widehat{T}/\q  \cong k[[z]]$ by the Cohen Structure Theorem.  Letting $\rho \colon A\to S \to \widehat{T}/\q$ be the composition of the natural homomorphisms gives the desired result.
 \end{proof}
 
 The previous lemma allows us to deduce an analog of the complex-analytic valuative criterion for ideal-theoretic integral dependence.

%: prop   alg curve crit
\begin{prop} \label{alg curve crit}Let $(A, \m, k)$ be the local ring of an algebraic variety over an algebraically closed field $k$, $I$ an ideal of $A$, and $h \in A$.  Then, $h \in \overline{I} \Leftrightarrow$ for every local homomorphism of $k$-algebras $\rho \colon A \to k[[z]]$ we have $\rho(h) \in \rho(I)k[[z]]$. 
\end{prop} 

\begin{proof} First suppose that $h$ is integral over $I$.  Let $\rho \colon A \to k[[z]]$ be a local homomorphism of $k$-algebras, Applying $\rho$ to an equation of integral dependence for $h$ over $I$ we see that $o(h) \ge o(I)$, where $o$ is the natural order function on $k[[z]]$, and consequently $\rho(h) \in \rho(I)k[[z]]$.

Conversely, assume that  $\rho(h) \in \rho(I)k[[z]]$ for every local homomorphism of $k$-algebras $\rho \colon A \to k[[z]]$.  Let $R = A[It]$ and $S = \overline{R}$.  Choose generators $g_1, \ldots , g_{\ell}$ of $I$. By Lemma \ref{lem: local char of int cl}  and Theorem \ref{prop: local I} it suffices to show that $h \in IS_{(g_it)} = g_iS_{(g_it)}$ for all $i$.  Suppose not.  Say $h \notin g_1S_{(g_1t)}$.  Then there exists a homogeneous prime $\q$ of $S$ not containing $g_1t$ that contracts to $\m$ and such that $h \notin g_1S_{\q}$.  Notice that $A[It] \cap \q \subset \m + It$.  By Going Up there exists a prime ideal $Q$ in $S$
containing $\q$ lying over $\m + It$.  Then, $S/Q \cong R/ \m + It \cong A/ \m = k$.  
Apply the preceding lemma to $S_{Q}, h$, and $g_1$  to obtain a local homomorphism of $k$-algebras $\rho \colon S_{Q} \to k[[z]]$ such that $\rho(h) \notin \rho(g_1)k[[z]] = \rho(I)k[[z]]$ since $IS_{Q} = g_1S_Q$ by virtue of the assumption that $g_1t \notin Q$.  Preceding this map by the natural local homomorphism $A \to S_Q$ gives a local homomorphism of $k$-algebras $\sigma \colon A \to k[[z]]$ such that $\sigma(h) \notin \sigma(I)k[[z]]$, a contradiction.  
\end{proof}
  
We return to the operation of relative closure in order to introduce our valuative criterion. The next proposition gives some basic facts about the relative closure.

\begin{prop} Let $\cO_{X,x}$ be the local ring of a complex analytic space.   Suppose that  $M\subset N\subset F$ are submodules of a free $\cO_{X,x}$-module $F$. 

\begin{enumerate} 
\item If $N=F, \; M \subset \m_xF$  then ${\overline {M_N}}$ is $\m_xF$. 

\item For every $N$, $\overline{M} \subset{\overline {M_N}} \subset \overline{N}$, with equalities  if $N\subset \overline{M}$.
\end{enumerate}
\end{prop}

\begin{proof} If $N=F$ is free, then for any curve $\phi$, $m_1 F'=m_1\phi^*(N) + \phi^*(M)$ , where $F'$ is the free module containing $\phi^*(N)$. 

To prove (2), let $h \in \overline{M}$. The valuative criterion for integral closure implies $\phi^*(h)  \in \phi^*(M)$ and hence $h \in \overline{M_N}$.  The second inclusion follows immediately from the valuative criterion for integral closure.
If  $N\subset \overline{M}$, then $\overline{M} = \overline{N}$ and both inclusions  must be equalities.
\end{proof}

The last proposition shows that the relative closure is in general larger than the integral closure. If $M$ and $N $ are ideals of finite colength, then the next pair of propositions explains why this is true. Moreover, they  show that we need only consider finitely many Rees valuations in computing the relative closure of 0-dimensional ideals $I \subset J$.  First we establish some notation.

Suppose $I\subset J$ are 0-dimensional ideals  in the local ring $\mathcal{O}_{X,x}$ of a complex analytic space.  Let $NB_J(X,x)$ denote the normalized blow-up of $X,x$ by $J$, with projection map $\pi_J$, and exceptional divisor $E_J$, with components $V_i$ and associated Rees valuations $v_i$. Let $\fa$ denote $[\pi^*_J(I):\pi^*_J(J)]$. Then we can form
$NB_{\fa}(NB_J(X,x))$, with projection map $\pi_{\fa}$ to $NB_J(X,x)$, and exceptional divisor $E_{\fa}$ with components $W_j$ and associated Rees valuations $w_j$. Then we have the following proposition.

%: rel clos ideal anal
\begin{prop} \label{rel clos ideal anal} Suppose $I\subset J$ are 0-dimensional ideals  in the local ring $\mathcal{O}_{X,x}$ of a complex analytic space and $h \in \cO_{X, x}$.  With notation as above,  $h\in \overline{I_J}$ if and only if $$v_i(h)\ge v_i(J) \;  \forall  i$$
and 
$$w_j(h)> w_j(J) \;  \forall  j.$$
\end{prop}

\begin{proof} Suppose $h\in \overline{I_J}$. Then $h\in \overline{J}$, so $v_i(h)\ge v_i(J)$ for all $i$.

Consider the components $W_j$ and their images in $E_J$. Suppose the image is either a component of $E_J$ or properly contained in a component of $E_J$ which is the image of a component of $E_{\fa}$.

Work at a generic point of the component $V_i$ of $E_J$ which is the image of a component of  $E_{\fa}$. We can write $h\circ\pi_J$ as  $(h\circ\pi_J/f_J\circ\pi_J)(f_J\circ\pi_J)$, where $h\circ\pi_J/f_J\circ\pi_J $ is holomorphic at our generic point, where $f_J$ is a generic element of $J$, since $h$ is in the integral closure of $J$. Since our component is in the image of a component of  $E_{\fa}$, it follows that the component is in $V(\fa)$. This implies that for any curve $\phi$ on $X$ with lift to the generic point of $V_i$, that $\phi^*(I)+m\phi^*(J)=m\phi^*(J)$. So, $h\in \overline{I_J}$ implies for such a $\phi$ that $(h\circ\pi_J/f_J\circ\pi_J)$ vanishes at the generic point of $V_i$; otherwise $o_{\phi}(h)=o_{\phi}(J)<o_{\phi}(I)$. So, $(h\circ\pi_J\circ\pi_{\fa}/f_J\circ\pi_J\circ\pi_{\fa})$ vanishes on any component of $E_{\fa}$ which maps to $V_i$, hence 
$w_j(h)> w_j(J)$ for all such components $W_j$.

So we may suppose that no component of $E_J$ that contains the image of $W_j$ lies in $V(\fa)$. On the other hand, every point of the image of $W_j$ lies in $V(\fa)$. It follows that for any curve $\phi$ with a lift to the image of $W_j$ that $\phi^*(I)+m\phi^*(J)=m\phi^*(J)$. So, $(h\circ\pi_J/f_J\circ\pi_J)$ must vanish at the point of the lift of $\phi$ over $x$, which again implies $w_j(h) >w_j(J)$.

Suppose $v_i(h)\ge v_i(J) \; \forall  i$ and $w_j(h) > w_j(J) \; \forall  j.$
By the first hypothesis we have $h\in \overline{J}$. 

Given a curve $\phi$ on $X,x$, denote the lift to $NB_J(X,x)$ by $\phi_J$ and by $\phi_{\fa}$ the lift to $NB_{\fa}(NB_J(X,x))$. Suppose $\phi_J(0)$ lies in the image of a component of $E_{\fa}$. Then $\phi_{\fa}(0)$ lies in a component of $E_{\fa}$.  Since $h\in \overline{J}$, in a neighborhood of $\phi_J(0)$, we can find a local generator, $f_J\circ\pi_J$ of $\pi_J^*(J)$, so that we can write $h\circ\pi_J$ as  $(h\circ\pi_J/f_J\circ\pi_J)(f_J\circ\pi_J)$.  Now consider
$ (h\circ\pi_J/f_J\circ\pi_J)\circ\pi_{\fa}$. This is a unit at $\phi_{\fa}(0)$ or it is not.
If it is a unit, then the ideal generated by $h\circ\pi_J\circ\pi_{\fa}$ agrees with the ideal generated by  $(\pi_J\circ\pi_{\fa})^*(J)$ locally. But then $w_j(h)= w_j(J)$  for all $j$ such that
$\phi_{\fa}(0)$  lies in $W_j$, which is a contradiction. So, $ (h\circ\pi_J/f_J\circ\pi_J)\circ\pi_{\fa}$ is not a unit at $\phi_{\fa}(0)$.

Hence,

$$o(h\circ\phi)=o(h\circ\pi_J\circ\pi_{\fa}\circ\phi_{\fa})> o( \phi_{\fa}^*((\pi_J\circ \pi_{\fa})^*(J) ))=o(\phi^*(J)).$$

Suppose $\phi_J(0)$ does not lie in the image of a component of $E_{\fa}$. Then $\phi_J(0)$ does not lie in $V({\fa})$, so in a neighborhood of $\phi_J(0)$,
$\pi_J^*(I)=\pi_J^*(J)$, hence 
$$o(h\circ\phi)=o(h\circ\pi_J\circ\phi_J)\ge o(\phi_J^*(\pi^*_J\*(J)))= o(\phi_J^*(\pi^*_J\*(I)))=o(\phi^*(I)).$$
This concludes our proof.
\end{proof}

Before presenting algebraic analogue we establish some notation.  Let $I \subset J $ be 
0-dimensional ideals  in the local ring $(A, \m , k) = \cO_{X, x}$ of an algebraic variety over an algebraically closed field $k$ and let $ \fa = I:J $.  Suppose that $I = (g_1 , \ldots , g_{\ell})$,  $J =(g_1 , \ldots , g_m)$ and $\fa =  (a_1, \ldots, a_t) $.  Let $R_i = \overline{A[J/g_i]} \; \: (i = 1, \ldots , m)$ and $S_{i j} = \overline{R_i[R_i \fa /a_j]} \; \: (i = 1, \ldots , m, \; j = 1, \ldots  , t)$.  

%: rel clos ideal alg
\begin{prop}  \label{rel clos ideal alg} Suppose $I \subset J $ are 0-dimensional ideals  in the local ring $(A, \m , k) = \cO_{X, x}$ of an algebraic variety over an algebraically closed field $k$ and let $ \fa = I:J $.   With notation as above, given $h \in A$, consider two set of Rees valuations: $\Sigma_1 =  \vv(J)$  and $\Sigma_2 = \bigcup_{i, j}\vv(\fa S_{i j})$. Then $h\in \overline{I_J}$ if and only if 
$$v(h) \ge v(J) \;  \forall v \in \Sigma_1$$ and $$v(h) > v(J) \;  \forall v \in \Sigma_2.$$
\end{prop}

\begin{proof}   First suppose that  $h\in \overline{I_J}$.
Since $\overline{I_J} \subseteq \overline{J}$ we have $v(h) \ge v(J) \;  \forall v \in \Sigma_1$.

Suppose that $v \in \Sigma_2$.  Then   $v$ is the valuation associated with the valuation ring $S_{\q}$ where $S = S_{i j} = \overline{R_i[R_i \fa /a_j]}$ for some indices $i, j$ and $\q$ is a minimal overprime of $\fa S$.
We must show that $v(h) > v(J)$. 

Now $I \subset \fa$ implies $I \subseteq \q \cap A$ and hence $\q \cap A = \m$ and $JS \subset \q$.
So $h \in \q$.  Let $JS_{\q} = g_iS_{\q}$ and $g = g_i$.

Choose any maximal ideal $\n$ of $S$  such that $S_{\n}$ and $S_{\n}/ \q S_{\n} $ are regular.  
Let $\overline{u_2}, \ldots , \overline{u_d}$ be a regular system of parameters for $S_{\n}/ \q S_{\n} $ and $z$ generate the height one prime $\q S_{\n}$ in the UFD  $S_{\n}$.  Set $u_1 = z$ so that  $u_1, \ldots , u_d$ is a regular system of parameters for $S_{\n}$.
 Write $g = fz^r$ where $r \in \N$ and $f \in S_{\n} \setminus \q S_{\n}$.  Replacing $\n$ if necessary we may and shall assume that  $f$ is a unit in $S_{\n}$.  Then consider the local homomorphism 
 $$\rho \colon A \to S_{\n} \to \widehat{S_{\n}}/(u_2 , \ldots , u_d) = k[[z]]$$
  and let $o(\; )$ denote  order with respect to $z$.  We have  $v(JS_{\q}) = r$ and $o(\rho(J)k[[z]]) =r$.    Write $hS_{\n} = h'z^s$ where $s \in \N$ and $h' \in S_{\n} \setminus \q S_{\n}$.  Again replacing $\n$ if necessary, we may and shall assume that $h'$ is a unit in $S_{\n}$.  Then $o(\rho(h)) = s = v(h)$.

For each generator of $I$ write $g_i = f_iz^{t_i}$ where $t_i \in \N$ and $f_i \in S_{\n} \setminus \q S_{\n}$.  Again replacing $\n$ as needed, we may and shall assume that each $f_i$ is a unit in $S_{\n}$. Thus
 for $t = \min\{t_i\}$ we have $t = v(IS_{\q}) = o(\rho(I)k[[z]])$.    Now $t > j$ since $I_{\q} \subsetneq J_{\q}$. Hence $\rho(h) \in \rho(I)k[[z]] + z \rho (J)k[[z]] = z \rho (J)k[[z]]$, which implies $o(\rho(h)) > o(\rho (J)k[[z]])$, i.e., $v(h) > v(J)$.
 
To prove the other direction assume that  $v(h) \ge v(J) \;  \forall v \in \Sigma_1$ and $v(h) > v(J)$ for all $v \in \Sigma_2$.  Let $\rho \colon A \rightarrow k[[z]]$ be a local homomorphism of $k$-algebras.  We wish to see that $\rho(h) \in \rho(I)k[[z]] + z \rho (J)k[[z]]$.  Since $v(h) \ge v(J) \;  \forall v \in \Sigma_1$ we know that $h \in \overline{J}$ and hence $\rho(h) \in \rho(J)k[[z]]$. So if $ \rho(I)k[[z]] = \rho(J)k[[z]]$ then $\rho(h) \in \rho(I)k[[z]] + z \rho (J)k[[z]] ]$.

Thus we may and shall assume that $ \rho(I)k[[z]] \subsetneq  \rho(J)k[[z]]$ and hence $ \rho(I)k[[z]] + z \rho (J)k[[z]] = z \rho (J)k[[z]]$. We must now show that $o(\rho(h)) > o(\rho(J)k[[z]])$.

Notice that $\rho$ lifts to $\rho_1 \colon R := \overline{A[J/g_i]}$ where $o(\rho(g_i)) = \min \{o(\rho( g_1)) , \ldots , o(\rho(g_m)) \}$.  Now $\rho(I)k[[z]] \subsetneq \rho ( J) k[[z]]$ implies $\rho_1 ( \fa ) \subset zk[[z]] =: \n$.

Further extend $\rho_1$  to $\rho_2 \colon S \rightarrow k[[z]]$, where $o(\rho_1(a_j)) = \min \{ o(\rho_1 (a_1)) , \ldots , o(\rho_1 (a_t)) \}$ and $S := \overline{R[R \fa /a_j]}$.  Let $\n_i = \rho_i^{-1}(\n) \; \; (i = 1, 2)$. 
 Since $\fa S \subset \n_2$, some minimal overprime $\q$ of $\fa S$ is contained in $\n_2$.  Hence $\fa \subset \q  \cap R \subset \n_1 = \n_2 \cap R$.
 
 Now $JR = g_iR$ and $h = \frac{h}{g_i} g_i$, where $ \frac{h}{g_i} \in R \subset S$. Consider $\frac{h}{g_i} \in S_{\n_2}$. If it is a unit, then $JS_{\n_2} = g_iS_{\n_2} = hS_{\n_2}$.  Further localizing at $\q$ we get $JS_{\q} = hS_{\q}$, contradicting the assumption that $v(h) > v(J) \; \forall v \in \Sigma_2$.  Thus $ \frac{h}{g_i} \in \n_2 S_{\n_2}$.  So $\rho(h) \in \rho(\n_2)\rho(g_i)k[[z]] \subset z\rho(J)k[[z]]$, as desired.
\end{proof}

Something similar but more complicated holds for relative closure of modules, so we postpone describing it now.

Before giving the construction of the modules which appear in our valuative criterion, we give some motivation in the analytic case. We are given an ideal $I$ and an element $h$ of $\overline{I}$, and we want to use curves to test if $h \in {}^*I$. We know that $h$ is in the weak subintegral closure if and only if  the blow-up of $X$ along the ideal $(I,h)$ is homeomorphic to the blow-up of $X$ along $I$, by the projection onto $B_I(X)$. So, we want to use curves to decide 
whether or not the map from the blow-ups are homeomorphic. 

In order for the blow-up by $(I,h)$ to map homeomorphically onto the blow-up by 
$I$ we need to require that for any two curves $\phi_1$, $\phi_2$ on $X$, whose lifts to  $B_I(X)$ map to the same point at $t=0$ in the fiber of over $\phi_i(0)$, 
then the lifts to
the blow-up by $(I,h)$ lift to the same point as well. If we have two curves, $\phi_1$ and $\phi_2$, we can treat them as a single curve by thinking of them as a curve on the product $Y:=X\times X$. So the modules we construct will be submodules of $\cO^2_{Y,y}$, where $y = (x,x)$. Here is the construction.
We have $\pi_i$, the projection of $X\times X$ onto the $i^{\rm th}$ factor. Consider the submodule of 
$\cO^2_{Y,y}$ generated by $\pi^*_1(I)\oplus\pi^*_2(I)$. This an interesting submodule for us. The diagonal submodule, $\Delta(I)$ is generated by elements $(h\circ\pi_1,h\circ\pi_2)$, $h\in I$,
where $\Delta \colon \cO_{X,x} \to \cO_{Y,y}^2$ is the injection of $\cO_{X}$-modules given by $\Delta(h)=(h \circ \pi_1,h \circ \pi_2)$. Given an ideal
$I \subset \cO_{X,x}$ then the pair of modules of interest are $\Delta(I)$ and $\pi^*_1(I)\oplus\pi^*_2(I)$, which we denote by $2I$.

%: val crit anal
\begin{theorem} \label{al crit anal} Let $\cO_{X,0}$ be the local ring of an a complex-analytic space. Suppose  $I$ is an ideal of finite colength in $\cO_{X,0}$, $h\in \cO_{X,0}$. Then, $h\in {}^*I$ if and only if $\Delta(h)\in \overline{\Delta(I)_{2I}}$.
\end{theorem}

\begin{proof} Suppose $h \in {}^*I$. Let $J$ denote the ideal generated by $(I,h)$. The assumption $h \in {}^*I$ implies that 
$B_J(X)$ is homeomorphic to $B_I(X)$ by the projection map. Suppose $\phi$ is a curve on $X\times X$ with components $\phi_1$ and $\phi_2$. Pick a set of generators $g_1, \ldots, g_k$
of $I$  and compose $\Delta(g_i)$ with  $\phi$; if there are $k$ generators, think of this as $k$ column vectors with two entries. 

Working mod $m_1 2I$, we can drop the terms of the first row of degree higher than the order of $\phi^*_1(I)$.  We can also truncate $g\circ\phi_2$, dropping terms of degree greater than the order of $\phi^*_2(I)$. Denote the truncated $k$-tuples by $(g\circ\phi_i)_T$.

There are now two cases.

Case 1: Suppose the lifts of $\phi_1$ and $\phi_2$ to $B_I(X)$, lift to different points over $0$. This is true if and only if the homogenous $k$-tuples $(g\circ\phi_i)_T$ are linearly independent. In turn, this means that the module 
$\phi^*(\Delta(I))+m_1\phi^*(2I)$ is equal to $\phi^*(2I)$. Hence $\Delta(h)\circ\phi\in \phi^*(\Delta(I))+m_1\phi^*(2I)$.

Case 2: Suppose the lifts of $\phi_1$ and $\phi_2$ to $B_I(X)$, lift to the same point over $0$. Then since $B_K(X)$ is homeomorphic to $B_I(X)$, it follows that the lifts to  $B_K(X)$ must lift to the same point in the fiber over $0$.
These assumptions imply that the tuples $(g\circ\phi_i)_T$ are linearly dependent, and the tuples  $((h,g)\circ\phi_i)_T$ are linearly dependent. Think of the last two tuples as the rows of a matrix, which evidently has rank 1. Then the column rank of the matrix must also be 1. Since $h\in \overline{I}$, it follows that the order of $h\circ\phi_i$ is no less than the order of $(g\circ\phi_i)_T$. Hence the $h$ column can be written in terms of the $g$ columns which implies the result.

Now suppose $\Delta(h)\in \overline{\Delta(I)_{2I}}$. First note that $h\in \overline{I}$. To see this just take $\phi=(\phi_1,0)$, where 
$\phi_1$ is arbitrary. In this case the condition of relative closure boils down to $h\circ\phi_1\in \phi_1^*(I)+m_1\phi_1^*(I)$ and it follows by Nakayama's lemma that $h\circ\phi_1\in \phi_1^*(I)$ which implies $h\in \overline{I}$.

This implies that $B_K(X)$ is finite over  $B_I(X)$ by the projection map, and since $I$ has finite colength, the projection is a homeomorphism, except possibly  when restricted to the fiber over $0$.

Suppose the map is not a homeomorphism; then we can find curves $\phi_1$ and $\phi_2$ such that the lifts of $\phi_1$ and $\phi_2$ to $B_I(X)$, lift to the same point over $0$, but the lifts to $B_K(X)$ lift to different points. 

Truncating as before, this implies the tuples $(g\circ\phi_i)_T$ are linearly dependent, while the tuples  $((h,g)\circ\phi_i)_T$ are linearly independent.  This of course means that it is impossible to write the $h$ column in terms of the $g$ columns, which contradicts the hypothesis. Since $I$ is 0-dimensional, it is regular. Then the conclusion of proof, showing $R(I+(h))$ is a subintegral extension of $R(I)$,follows from  Proposition \ref {proposition homeo to subintegral}.
\end{proof}

The algebraic analog of the valuative criterion for the local ring of a complex algebraic variety follows.  The proof unfolds in pretty much the same fashion.  First we need some notation.

Let $(A,\m, k)$ be the local ring of an algebraic variety over an algebraically closed field $k$ of characteristic 0 and $I$ be a 0-dimensional ideal in 
$A$.   
Let $B = A \otimes _{k} A$ and let $\lambda_i \colon A \to B$ denote the natural maps defined by $\lambda_1(a) = a \otimes 1$ and $\lambda_2(a) = 1 \otimes a$.  Consider the submodule of $B^2$ generated by $\lambda_1(I)B \oplus \lambda_2(I)B$; we will denote this submodule by $2I$.   Let $\Delta \colon A \to B^2$ be defined by $\Delta(a) = (\lambda_1(a), \lambda_2(a))$.

%: val crit alg
\begin{theorem}  \label{val crit alg} Let $(A,\m, k)$ be the local ring of an algebraic variety over an algebraically closed field $k$ of characteristic 0,  $I$  a 0-dimensional ideal in $A$, and 
$h\in A$.  With notation as above,  $h\in {}^*I$ if and only if $\Delta(h)\in \overline{\Delta(I)_{2I}}$.
\end{theorem}

\begin{proof}  Let $X = \spec(A)$ and let $x$ denote the closed point corresponding to $\m$.  If $\dim (A) = 0$ there is  nothing to prove.  So assume that $\dim (A) > 0$. 
Suppose $h \in {}^*I$. Let $J$ denote the ideal generated by $(I,h)$. The assumption $h \in {}^*I$ implies that  $A[It] \subset A[Jt]$ is a weakly subintegral extension and hence
$B_J(X)$ is homeomorphic to $B_I(X)$ by the induced map. Suppose $\rho \colon A \otimes_k A \to k [[z]]$  is a local homomorphism of $k$-algebras and let $\rho_i \colon A \to k [[z]]$ be the compositions with the natural maps 
$\lambda_i: A \to A \otimes_k A$.  Pick a set of regular generators $g_1, \ldots, g_{\ell}$
of $I$  and look at their images in $\rho(I)k[[z]]$; since there are $\ell$ generators, think of this as a $\ell$-tuple $\rho_1(g)$. 
Working mod $z\rho_1(I)k[[z]]$, the terms of  $\rho_1(g)$ of degree higher than the order of $\rho_1(I)k[[z]]$ become 0. Denote the image modulo $z \rho_1(I)k[[z]]$ of the $\ell$-tuple by $\rho_1(g)_T$. We can also ``truncate" $\rho_2(g)$, by reading modulo $z \rho_2(I)k[[z]]$; denote this image by $\rho_2(g)_T$.  Let $\rho(g)_T$ denote the $2 \times \ell$ matrix with rows $\rho_i(g)_T$.

There are now two cases.

Case 1:  The map $\rho_1$ has a unique extension to a map from $A[I/g_1]$ to $k[[z]]$, where we have reindexed $g_1, \ldots , g_{\ell}$ so that $\rho_1(I)k[[z]] = \rho_1(g_1)k[[z]] $.  In turn we can extend  $\rho_1$ to a map $\widetilde{\rho_1}$ on the Rees algebra $A[It]$ by setting $\widetilde{\rho_1}(g_1t) = 1$ and $\widetilde{\rho_1}(g_jt) = \rho_1(g_j)/\rho_1(g_1) $. This is well defined.

Suppose that $\rho_1(g_j) = (a_j + zg'_j)z^{e_j} $, where $a_j \in k, g'_j \in k[[z]], e_j \in \N (\; j = 1, \ldots ,  \ell)$.  Notice that $\widetilde{\rho_1}^{-1}(zk[[z]]) = \m + (a_jg_1t-a_1g_jt \mid e_j = e_1) + (g_jt \mid e_j > e_1)$.  Additionally, we see that  $\rho_1(g)_T \cong (a_1z^{e_1}, \ldots , a_{\ell}z^{e_{\ell}}) \cong z^{e}(\delta_{e_1 e_1}a_1, \ldots, \delta_{e_{\ell} e_1}a_{\ell}) \pmod{ z\rho_1(I) k[[z]]}$, where $e = e_1$.   Something similar holds for $\rho_2$ resulting in $\rho_2(g)_T \cong z^f(\delta_{f_1 f}b_1 , \ldots , \delta_{f_{\ell} f}b_{\ell})$.

Suppose that $\widetilde{\rho_1}^{-1}(z k[[z]])  \ne  \widetilde{\rho_2}^{-1}(z k[[z]]) $.
This is true if and only if the $\ell$-tuples of complex numbers $z^{-e}\rho_1(g)_T$ and $z^{-f}\rho_2(g)_T$ are linearly independent. In turn, this means that the module 
$\rho^{(2)}(\Delta(I)) k[[z]]+z\rho^{(2)}(2I)k[[z]]$ is equal to $\rho^{(2)}(2I) k[[z]]$.   Hence $\rho^{(2)}(\Delta(h))  \in \rho^{(2)}(\Delta(I) k[[z]]+z\rho^{(2)}(2I) k[[z]]$.

Case 2: Suppose that $\widetilde{\rho_1}^{-1}(zk[[z]])  = \widetilde{ \rho_2}^{-1}(zk[[z]]) $.  Since $A[It] \subset A[Jt]$ is a weakly subintegral extension, if we  extend each map $\rho_i$ $A[Jt]$ in stages, as above, then the contractions of $zk[[z]]$ to $A[Jt]$ via $\rho_1$ and $\rho_2$ must be equal.  
This implies that the matrices $(z^{-e}\rho_1(g)_T,z^{-f}\rho_2(g)_T) $ and $(z^{-e}\rho_1((h,g))_T, z^{-f}\rho_2((h,g))_T)$ have row rank 1. Then the column rank of the matrix must also be 1. Recall that  $h\in \overline{I}$, implies that the order of $z$ in  $\rho_i( h)$ is no less than the order in $\rho_i(g)$. Hence the $h$ column can be written in terms of the $g$ columns which implies the result.

Now suppose $\Delta(h)\in \overline{\Delta(I)_{2I}}$. First note that $h\in \overline{I}$ by Proposition \ref{alg curve crit}. To see this just take $\rho=(\rho_1,\eta)$, where 
$\rho_1$ is arbitrary and $\eta \colon A \to A/\m \to k[[z]]$ is the composition of the natural maps. In this case the condition of relative closure boils down to $\rho_1(h) \in \rho_1(I)k[[z]] + z\rho_1(I)k[[z]] = \rho_1(I)k[[z]] $ and hence  $h\in \overline{I}$.

This implies that $B_J(X)$ is finite over  $B_I(X)$ by the projection map, and since $I$ is 0-dimensional, the projection is a homeomorphism, except possibly  when restricted to the fiber over $0$.

Suppose the map is not a homeomorphism.  Then there is a closed point in $B_I(X)$, lying over $x \in X$, with two preimages in $B_J(X)$.  Hence there is a generator $g$ of $I$ and two maximal ideals $\n_1$ and $\n_2$ of $A[J/g]$ that contract to the same maximal ideal $\n$ of $A[I/g]$.   

We thus have $k$-algebra maps $\rho_i \colon A[J/g] \to k [[z]] \; (i = 1,2)$ such that $\rho_i^{-1}(zk[[z]]) = \n_i \; (i=1,2)$.
Here are the details on the finding the maps. First take height one prime ideals $\q_i \subset A[J/g]$ contained in $\n_i$.  Mod out by $\q_i$, localize at $\n_i$, normalize, complete, and then take an analytic branch.  Using the composition of the map from $A[J/g]$ to the analytic branch does the job.

With notation as above, this implies the rows of $(z^{-e}\rho_1(g)_T, z^{-f}\rho_2(g)_T)$ are linearly dependent, whereas the rows of  $(z^{-e}\rho_1(h,g)_T, z^{-f}\rho_2(h,g)_T)$  are linearly independent.  This of course means that it is impossible to write the $h$ column in terms of the $g$ columns, which contradicts the hypothesis.

Hence the blow-ups are homeomorphic.  Hence the extension $R := A[It] \subset R^{\sharp} : =A[Jt]$ is weakly 
subintegral by Lemma \ref{lem homeo to subintegral}. 
\end{proof}

 If the ideal $I$ is 0-dimensional,
this criterion is easy to work with because you only need to work at one point. The above theorem holds when $I$ is not 0-dimensional, but one must assume that $\Delta(h)\in \overline{\Delta(I)_{2I}}$ holds at every point of the diagonal of $V(I)\times V(I)$.

%: bibliography

\end{document}